\documentclass[11pt,reqno]{amsart}
\usepackage[babel]{csquotes}
\usepackage{enumitem}
\usepackage{amsmath,amsthm,amssymb,mathrsfs,amsfonts,verbatim,enumitem,color,leftidx}
\usepackage{mathabx}
\usepackage[left=2.9cm,right=2.9cm,top=3.3cm,bottom=3.4cm,a4paper]{geometry}

\usepackage{etoolbox} 
\usepackage{bbm}
\usepackage[all,tips]{xy}
\usepackage{graphicx,ifpdf}
\usepackage{stmaryrd}
\ifpdf
\fi

\usepackage[colorlinks]{hyperref}
\hypersetup{
linkcolor=blue,          
citecolor=green,      
}

\usepackage{tikz}

\newtheorem{thm}{Theorem}[section] 
\newtheorem{lem}[thm]{Lemma}
\newtheorem{coro}[thm]{Corollary}
\newtheorem{prop}[thm]{Proposition}

\theoremstyle{definition}
\newtheorem{defn}[thm]{Definition}

\theoremstyle{remark}

\numberwithin{equation}{section}

\definecolor{esperance}{rgb}{0.0,0.5,0.0}

\newcommand{\R}{\mathbb{R}}

\newcommand{\X}{ \widetilde{X}}

\newcommand{\G}{\Gamma}

\newcommand\blfootnote[1]{%
  \begingroup
  \renewcommand\thefootnote{}\footnote{#1}%
  \addtocounter{footnote}{-1}%
  \endgroup
}

\newcommand{\cB}{\mathcal{B}}

\newcommand{\cE}{\mathcal{E}}

\newcommand{\cH}{\mathcal{H}}

\newcommand{\cM}{\mathcal{M}}

\newcommand{\cP}{\mathcal{P}}

\newcommand{\onto}{\xymatrix{\ar@{>>}[r]&}}

\begin{document}
\title{Martin boundary of Brownian motion on Gromov hyperbolic metric graphs}
\author{Soonki Hong}
\address{Department of Mathematical Sciences, Seoul National University, Seoul 151-747}
\email{soonki.hong@snu.ac.kr}
\author{Seonhee Lim}
\address{Department of Mathematical Sciences and Research Institute of Mathematics, Seoul National University, Seoul 151-747}
\email{slim@snu.ac.kr}

\maketitle
\begin{abstract} Let $\X$ be a locally finite Gromov hyperbolic graph whose Gromov boundary consists of infinitely many points and with a cocompact isometric action of a discrete group $\G$.
We show the uniform Ancona inequality for the Brownian motion which implies that the $\lambda$-Martin boundary coincides with the Gromov boundary for any $\lambda \in [0, \lambda_0],$ in particular at the bottom of the spectrum $\lambda_0$. \end{abstract}

\section{Introduction}
The theory of Brownian motion has been established as a central theme in mathematical physics and probability theory for a long time. Classically, Brownian motion is defined as a Wiener process, i.e. a continuous time Markov process of which the probability density $p(t,x,y)$ of going from $x$ to $y$ at time $t$ is given by the fundamental solution of the heat equation (see Definition \ref{Def2.6}). Thus the probability density $p(t,x,y)$ can be defined only when the Laplace operator is chosen. For example, the usual choice of Laplace operator on Riemmanian manifolds is Laplace-Beltrami operator. One branch of studies on Brownian motion has been developed through Dirichlet form (see \cite{F}, \cite{Si} and references therein and \cite{AR}).

Dirichlet form enables us to define the Laplacian without using partial derivatives, thus it is suitable to study the heat equation on the spaces that are not differentiable manifolds. Various authors have verified that Laplacians related to Dirichlet forms have properties similar to Laplace-Beltrami operators of manifolds. For example, the maximal principle \cite{W}, \cite{HK},  stochastic completeness  \cite{St1}, \cite{W}, \cite{KL}, the spectrum of Laplacian \cite{SW}, \cite{SW2}, \cite{HK}, and the existence of harmonic functions \cite{LSV} have been investigated. The theory of Brownian motions related to Dirichlet forms on metric spaces has also been studied (see \cite{EF}, \cite{Gr}, \cite{PS}, \cite{BSSW} \cite{St4}). 
 
Let $\X$ be a locally finite Gromov hyperbolic metric graph with Gromov boundary consisting of infinitely many points. Suppose that there is a geometric (i.e. proper and cocompact) isometric action of a discrete group $\G$. Note that $\X$ is not necessarily a tree.

 Using a Dirichlet form $\mathcal{E}$ whose domain is the Sobolev space $W^1(O)$ on a precompact open set $O \subset \X$, we define the graph version of the Laplacian $(\Delta,Dom_{\X}(\Delta))$ on the space $\X$ and as well as on any open set $O \subset \X$. 
 
 We recall the existence and the smoothness of the heat kernel $p(t,x,y)$ on $\X$. For the existence, we recall two properties, namely  the doubling property \eqref{2.2} and Poincar\'e inequality \eqref{2.3}. The authors of \cite{BSSW} proved that strip complexes including graphs satisfy these properties. Applying the results in \cite{St2, St3}, one obtains the heat kernel of $\X$. The smoothness of the heat kernel of $\X$ follows from \cite{BSSW}.

We then consider the $\lambda$-Green function $$G_\lambda(x,y):=\int_0^\infty e^{\lambda t} p(t,x,y)dt$$ for any two distinct points $x,y\in \X$. 
Let $\lambda_0$ be the bottom of the spectrum of Laplacian, which depends on the group $\Gamma$. The group $\Gamma$ is non-amenable if and only if the bottom of the spectrum of the Laplacian is positive (\cite{SW} Theorem 8.5). The $\lambda$-Green function converges for all $\lambda \in [0,\lambda_0)$. Using the result in \cite{LSV}, we show the existence of a positive $\lambda$-harmonic function on the graph for any $\lambda\in[0,\lambda_0]$.  Following the proof in \cite{LS}, we prove the convergence of the $\lambda_0$-Green function.

Our main result is the uniform Ancona inequality, which we call Ancona-Gou\"ezel inequality on hyperbolic graphs (see \cite{GL} and \cite{G} for random walks).

\begin{thm}\label{thm:1.1} Let $\X$ be a locally finite topologically complete Gromov hyperbolic metric graph with the Gromov boundary consisting of infinitely many points. Suppose that a group $\Gamma$ acts isometrically and geometrically on $\X$. Let $l_m$ be the minimal edge length of $\X$.  There exists a constant $C$ such that for all $\lambda \in [0,\lambda_0]$ and for three points $x,y$ and $z$ on the same geodesic $[x,z]$ with $d(x,y)\geq 1$ and $d(y,z)\geq 1$,
\begin{equation}\label{1.1.0}
C^{-1}G_\lambda(x,y)G_\lambda(y,z)\leq G_\lambda(x,z)\leq CG_\lambda(x,y)G_\lambda(y,z).
\end{equation}
\end{thm}

The most non-trivial part is the uniformity of the constant $C$ in the inequality \eqref{1.1.0} on $\lambda$, which implies the inequality \eqref{1.1.0} for $\lambda=\lambda_0$. The Brownian motion associated with a Dirichlet form on a graph is a Hunt process, in particular a strong Markov process. Using strong Markov properties of Brownian motion, we show that the relative $\lambda$-Green function $G_\lambda(x,z:B(x,r)^c)$ decays super-exponentially fast, from which the uniform Ancona-Gou\"ezel inequality follows, an idea due to S. Gou\"ezel \cite{G}. 
 
The $\lambda$-\emph{Martin kernel} $K_\lambda$ of $\X$ is defined as follows:
$$K_\lambda(x_0,x,y)=\frac{G_\lambda(x,y)}{G_\lambda(x_0,y)}.$$
The $\lambda$-\emph{Martin boundary} is the boundary of the image of the embedding defined by $y\mapsto K_\lambda(x_0,\cdot,y)$ on $\X$. Using Ancona-Gou\"ezel inequality for $\lambda \in [0, \lambda_0]$, we show the next main theorem.

 \begin{thm}\label{1.2}Let $\X$ be a locally finite complete Gromov hyperbolic metric graph with the Gromov boundary consisting of infinitely many points. Suppose that a group $\Gamma$ acts isometrically and geometrically on $\X$. The $\lambda$-Martin boundary of $\X$ coincides with the Gromov boundary for all $\lambda\in[0,\lambda_0]$.\end{thm}
Our motivation to prove Theorem \ref{1.2} is the local limit theorem, which is an important question in the study of Brownian motions: does there exist a function $c:\widetilde{X}\times\widetilde{X}\rightarrow \mathbb{R}$ such that for two distinct points $x,y \in \widetilde{X}$,
$$\lim_{t\rightarrow\infty}t^{3/2}e^{\lambda_0 t}p(t,x,y)=c(x,y)?$$
 
   The proof of the local limit theorem for random walks on hyperbolic groups or Brownian motions on Riemannian manifolds uses various strategies (see \cite{Bo}, \cite{GL}, \cite{G} and \cite{LL}). In particular, in \cite{LL}, one uses Gibbs measure associated to a pressure which is defined using $\lambda$-Green function.  We expect that Theorem~\ref{1.2} will enable us to apply thermodynamics formalism of the geodesic flow on a hyperbolic graph and Gibbs measures on the $\lambda_0$-Martin boundary.

The article is organized as follows. In Section \ref{2}, we recall the definition of the Laplacian and the existence of the heat kernel on graphs, which are based on results of \cite{St1, St2, BSSW}.
In Section \ref{sec:3}, we observe properties of positive $\lambda$-harmonic functions. Using these properties, we show that if the group $\Gamma$ is non-amenable, then the $\lambda_0$-Green function converges.
In Section \ref{4}, we first prove Ancona-Gou\"ezel inequality \eqref{1.1.0} (Theorem \ref{lAncona} and Theorem \ref{rancona}). Using the inequality \eqref{1.1.0}, we show Theorem \ref{1.2}.

 \section{Preliminaries: Laplacian and Heat kernel}\label{2}
 Let $(\X,d)=((V,E),d)$ be a locally finite connected topologically complete Gromov hyperbolic metric graph. Let $B(x,r)$ and $S(x,r)$ be the ball and the sphere of radius $r$ centered at $x$, respectively. Denote by $diam (A)$ the diameter of a subset $A$ of $\X$. Fix an orientation of the edges on $\X$. Denote by $i(e)$ and $t(e)$ the initial vertex and the terminal vertex of an edge $e$, respectively. Denote by $e^o$ the open edge of $e$ and denote by $l_e$ the length of  $e$. The Gromov boundary $\partial \X$ of $\X$ is the set of equivalence classes of the geodesic rays up to bounded Hausdorff distance. Suppose that the cardinality of $\partial \X$ is infinite.

  Let $\Gamma$ be a non-amenable discrete group. Suppose that $\G$ acts isometrically and geometrically (i.e. properly and cocompactly) on $\X$, with the quotient space $X= \X/\Gamma$.  Using the barycenter subdivision if necessary, we may assume that $\Gamma$ acts without inversions. Since $X$ is compact and locally finite, the lengths of the edges of $\X$ are bounded above and below. Denote $l_m=\underset{e\in E} \min\, l_e$ and $l_M=\underset{e\in E}\max\, l_e.$ 
Fix a connected fundamental domain $F$ of $\G$ in $\X$. We denote by $d_\Gamma$ the word distance of $\Gamma$  with respect to the generating set $S=\left\{g \in \Gamma|\overline{F} \cap g.\overline{F} \neq \phi \right\}$, which is finite (\cite{BH}, Proposition I.8.19). 
  
In this section, we define the Laplacian on $\X$ using Dirichlet forms and discuss the existence and the smoothness of the heat kernel on $\X$. 

\subsection{Dirichlet form and Laplacian on graphs}\label{sec:2.1}For a function $f$ on $\X,$ let $f|_e$ be its restriction on $e^o$, which we will often consider as a function on $(0,l_e)$. For an open set $O \subset \X$, define $O^o:=\underset{e\in E}\bigcup (e^o\cap O)=O \backslash V$. We consider the derivative $f'$ of $f$ as the function on $\X^o$ when the function $(f|_e)'$ on $e^o$ exists for all $e\in E$. 
\begin{defn}
Let $O$ be an open set. Denote by $C^\infty(O)$ the vector space of continuous functions whose restriction on $O^o$ satisfies the following property:
 for any edge $e$ intersecting $O$ and any integer $k>0$, the partial derivative $(f|_{e})^{(k)}$ is continuous on $e^o\cap O$ and 
\begin{equation}\label{a}
\sup_{} \{|(f|_{e})^{(k)}(x)|: x \in e^o\cap O\}< \infty .
\end{equation}
Let $C_{c}^\infty(O)$ be the space of compactly supported functions in $C^{\infty}(O)$.
\end{defn} 
 
 Given $f \in C^\infty(\X)$, any derivative of $f|_e$ on $e^o$ can be continuously extended to $e$. Note that the $k$-th derivative $(f|_e)^{(k)}(v)$ at a vertex $v$ depends on $e$.  
 
 Using Lebesgue measure $ds$ on $\mathbb{R}$, we define Lebesgue measure $\mu$ on $ \widetilde{X}$ as follows: for any measurable function $f$ on $\X$,
$$
\displaystyle \int_{ \widetilde{X}} fd\mu := \displaystyle \sum_{e \in E} \int_0^{l_e} f(e_s)ds,
$$
where $e_s$ is a point on $e$ with $d(i(e),e_s)=s$.

 Denote by $W^1(e^o)$ the subspace of $L^2(e^o,ds)$-functions whose first weak derivative is also in $L^2(e^o,ds)$. 
\begin{defn}\label{def:2.2} Let $O$ be a connected open subset of $\X$. 
\begin{enumerate}
\item[(1)] The \emph{Sobolev space} $W^1(O)$ is the set of functions such that
\begin{enumerate}
\item[(i)] for every $f \in W^1(O)$, $f|_{e^o\cap O} \in W^1(e^o\cap O,ds),$
\item[(ii)] $|| f'||_{L^2(O)}^2=\displaystyle \int_O || f'||^2d\mu < \infty.$
\end{enumerate}
 Denote $||f||_{W^1(O)}:=(||f||_{L^2(O)}+||f'||_{L^2(O)})^{\frac{1}{2}}.$ Let $W_0^1(O)$ be the closure of $C_c^\infty(O)$ in $W^1(O)$. The vector space $W^1_{loc}(O)$ is the space of functions $f$ such that for any compact set $K\subset O$, there exist a function $g\in W^1(O)$ with $g|_K=f|_K$. 

\item[(2)] A symmetric form $\mathcal{E}$ on $W^1(O)$ is defined, for all $f,g \in W^1(O)$, by
\begin{displaymath}
\displaystyle \mathcal{E}(f,g):=\int_Of'g'd\mu=\displaystyle \sum_{\substack{e\in E\\ e\cap O\neq \phi}}\int_{e\cap O} (f|_e)'(g|_e)' ds.
\end{displaymath}
\end{enumerate}
\end{defn}

 The symmetric form $\mathcal E$ on $W_0^1(O)$ is a strongly local regular Dirichlet form and $(\mathcal{E}, W_0^1(\X))$ coincides with $(\mathcal{E}, W^1(\X))$ (\cite{BSSW} Theorem 3.29 and 3.30). See the Appendix for definition.

An operator $(A,Dom(A))$ is \emph{non-negative definite} and \emph{self-adjoint} if for any $u\in{Dom(A)}$, $(Au,u)\geq 0$ and $A$ has a transpose operator $A^t$ such that $Av=A^t v$ for all $v\in Dom(A)$ and $Dom(A)=Dom(A^t)$.

\begin{defn}\label{Def2.6} Let $O$ be a precompact connected open set.

\begin{enumerate}\item[(1)] The domain $Dom_O(\Delta)$ of Laplacian $\Delta$ on an open set $O$ is the space of functions $f$ in $W^1_0({O})$ for which there exists a constant $C_f$ such that for all $g \in W^1_0({O}),$ $|\mathcal{E}(f,g)| \leq C_f ||g||_{L^2({O})}$. Denote $Dom(\Delta)=Dom_{ \widetilde{X}}(\Delta)$.
\item[(2)] Let $f\in Dom_O(\Delta)$. By Riesz representation theorem, there exists a unique function $h \in L^2(O)$ such that \begin{displaymath} \mathcal{E}(f,g) = \displaystyle  -\int_{O} h\,g\,d\mu.\end{displaymath}
We define Laplacian $\Delta f$ of $f$ to be the function $h$. 
The operator $-\Delta$ is a non-negative definite self-adjoint operator defined on $Dom_O(\Delta) \subset W^1_0(O)$. The spectrum of $-\Delta$ consists of non-negative real numbers.
\end{enumerate}
\end{defn}

\subsection{Heat kernel on graphs}\label{sec:2.3} In this section, we first define the heat kernel of a graph. Using the general theory of Dirichlet forms, we will obtain the existence and the smoothness of the heat kernel of $\X$.
\begin{defn}\label{Def2.6} 
The \emph{heat kernel} of a graph $\X$ is a fundamental solution of the heat equation i.e. a continuous function $p$ such that $\Delta_x p(t,x,y)= \frac{\partial}{\partial t} p(t,x,y) $ and $p(t,x,y) \to \delta_{x-y}$ as $t \to 0$.
\end{defn}
For our space $\X$, the heat kernel will be unique (see Theorem \ref{thm: 2.10}).
Note that $\Delta f = f''$ on each open edge.
The existence of the heat kernel was proved by K.T. Sturm for spaces satisfying the doubling property \eqref{2.2} and Poincar\'e inequality \eqref{2.3}, which were proved by \cite{BSSW} for strip complexes including locally finite metric graphs.
\begin{thm}\cite{St2, St3, BSSW}\label{hk}
Suppose that $\widetilde{X}$ is a locally finite topologically complete metric graph. Then there exists a non-negative function $p(t,x,y)$ satisfying the following:  
\begin{enumerate}
\item[(1)] $\displaystyle P_tf(x):=e^{\Delta t}f(x)=\int_{\X}p(t,x,y)f(y)d\mu(y)$ for any bounded function $f$ on $\X$.
\item[(2)] $p(t,x,y)=p(t,y,x).$ 
\item[(3)] $\displaystyle \int_{ \widetilde{X}}p(t,x,y)p(s,y,z)d\mu(y)=p(t+s,x,z).$
\item[(4)] The function $t\mapsto p(t,x,y)$ is in $C^\infty((0,\infty))$ for any $x,y\in X$.
\end{enumerate}
\end{thm}
By part (1) of the above theorem, $p(t,x,y)$ is the heat kernel.
Let us recall the doubling property \eqref{2.2} and Poincar\'e inequality \eqref{2.3}: for any compact set $K\subset \X$, there exist $r_K>0, C_K \geq 1, P_ K \geq 1$ such that for all $x\in K$ and $r\in(0,r_K)$, 
\begin{equation}\label{2.2}
\mu(B(x,2r))\leq C_K \mu(B(x,r)),
\end{equation}
and for any $f\in W^1(B(x,r))$, 
\begin{equation}\label{2.3}
\int_{B(x,r)}|f-f_{B(x,r)}|^2 d\mu=P_Kr^2\int_{B(x,r)}|f'|^2d\mu,
\end{equation}
where $f_{B(x,r)}=\frac{1}{B(x,r)}\int_{B(x,r)}fd\mu$. For our space $\X$, by compactness of $X$, the above properties are satisfied globally, i.e. the constants $r_K, C_K,P_K$ can be chosen independently of $K$. Under this condition, \cite{BSSW} showed the positivity of the heat kernel.
\begin{thm} \emph{(\cite{BSSW} Theorem 4.6)}\label{positive} Let $\X$ be a topologically complete graph. If there are global constants $r,C,P$ for \eqref{2.2} and \eqref{2.3}, then the heat kernel $p(t,x,y)$ is positive.
\end{thm}

Another important property which we will use in the proof of Harnack inequality (Corollary \ref{G2}) is Kirchhoff law.
\begin{defn} \label{weaksol} Let $O$ be a connected open set of $\X$.
Define $D^\infty(O) \subset C^\infty(O)$ to be the subspace of functions $f$ that satisfies $(f|_e)^{(2k)}(v)=(f|_{e'})^{(2k)}(v)$ for all  integer $k>0$ and edges $e$ and $e'$ with $e\cap e'=\{v\}$ and Kirchhoff's law, i.e. for any vertex $v\in O$ and positive integer $k$, 
 \begin{equation}\label{Kirchhoff}
 \displaystyle \sum_{v=i(e)} (f|_e)^{(2k+1)}(v) =\sum_{v=t(e)} (f|_e)^{(2k+1)}(v).
 \end{equation}
\end{defn}

\begin{thm} \emph{(\cite{BSSW} Theorem 5.23 and Theorem 7.5)} The function $y\mapsto p(t,x,y)$ is positive and it is an element of $D^\infty(\X)$ for any $t\in \mathbb{R}$ and $x\in \X$.
\end{thm}
  
The stochastic completeness of the heat kernel follows from the hyperbolicity of $\X$ and the cocompactness of the action of $\Gamma$  together with a result of K.T. Sturm.

\begin{thm}\emph{(\cite{St1} Theorem 4)}\label{thm: 2.10} If $\widetilde{X}$ is topologically complete and for all $x\in \widetilde{X}$,
 \begin{equation}\label{volumeg}\displaystyle \int_1^\infty \frac{r}{\ln \mu(B(x,r))}dr =\infty,\end{equation}
 then the solution of the bounded Cauchy problem on $(0,T)\times\widetilde{X}$ is unique. In particular, $$e^{t \Delta}1=\int_{\X}p(t,x,y)d\mu(y)=1.$$
 \end{thm} 
\begin{coro}\label{Brooks} The heat kernel $p(t,x,y)$ is stochastically complete, i.e. 
\begin{displaymath}
 \int_{ \widetilde{X}}p(t,x,y)d\mu(y)=1.
\end{displaymath} 
\end{coro}
\begin{proof}  Since $\G$ acts cocompactly on a Gromov hyperbolic space $\X$, there is a constant $C$ such that $\mu(B(x,r))$ is bounded above by $e^{Cr}$, thus the equation \eqref{volumeg} holds.
\end{proof}
\section{Harnack inequality and $\lambda$-harmonic functions}\label{sec:3}
\subsection{Harnack inequality for graphs} \label{sec:3.1}
  In this section, we show the graph version of Harnack inequality, which is the analog of the result of Cheng and Yao (\cite{CY}). Denote the counting measure of a discrete subset $Y\subset \X$($A\subset \G$, resp.) by $|Y|$ ($|A |$, resp.). 
    \begin{defn} Let $O$ be an open subset of $ \widetilde{X}$. A function $f$ in $W_{loc}^1(O)$ is $\lambda$-\emph{harmonic} on $O$ if  $f$ is a weak solution of the following equation: for all $g\in W_c^1(O),$ 
\begin{equation}\label{weaksolution}
\mathcal{E}(f,g)=\lambda(f,g).
\end{equation}
\end{defn}
   \begin{lem}\label{bwws2}
Let $O$ be a precompact open subset of $ \widetilde{X}$. Any $\lambda$-harmonic function $f$ on $O$ is contained in $D^\infty(O)$.
\end{lem}
\begin{proof} The restriction of $f$ on an open edge $e^o\cap O$ is smooth. Thus $\Delta f|_{e}(x)$ coincides with $(f|_{e})''(x)$. Suppose that the open set $O$ contains a vertex $v$. Choose $r>0$ with $B(v,r)\subset O$. Let $e_{\epsilon}$ be a point in $e$ satisfying $d(v,e_\epsilon)=\epsilon$. 
We obtain
\begin{equation}\label{smooth} 
\nonumber\left|(f|_e)'(e_{\epsilon_1})-(f|_e)'(e_{\epsilon_2})\right|=\left|\int_{\epsilon_1}^{\epsilon_2}(f|_e)''(e_s)ds\right|=\lambda \left|\int_{\epsilon_1}^{\epsilon_2}(f|_e)(e_s)ds\right|
\end{equation}
Since $f|_e$ is continuous, $(f|_e)'(e_{1/n})$ is a Cauchy sequence and $(f|_e)'(v)$ exists. By the integration by part and the definition of Laplacian, for all $g\in C^\infty_c(O)$,  
\begin{equation}\label{kh}
\begin{split}
 -&\int_{O} f'g' d\mu=\int_O  \Delta f g d\mu=\sum_{e\cap O\neq \phi} \int_{0}^{l_e} (f|_e)''(e_s) (g|_e)(e_s) ds\\
=&\sum_{e\cap O\neq \phi}((f|_e)'(t(e))g(t(e))-(f|_e)'(i(e))g(i(e)))-\sum_{e\cap O\neq \phi} \int_{0}^{l_e} (f|_e)'(e_s) (g|_e)'(e_s) ds\\
=& \sum_{v\in O} \left\{\sum_{\substack{e\in E\\ t(e)=v}} (f|_e)'(v)g(v)-\sum_{\substack{e\in E\\i(e)=v}}(f|_e)'(v)g(v)\right\}-\int_{O} f'g' d\mu.
\end{split}
\end{equation}
By \eqref{kh}, we have $\underset{v=t(e)} \sum(f|_e)'(v)-\underset{v=t(e)}\sum(f|_e)'(v)=0$ when the support of $g$ is in $B(v,r)$ and $g(v)=1$. Since $f$ is $\lambda$-harmonic, $(f|_e)^{(2k+1)}(v)=\lambda^k(f|_e)'(v)$ for any $k\geq 0$ and $v\in V$, and $f$ satisfies the Kirchhoff's law.
\end{proof}
\begin{prop}\label{gphar} Fix $r,l>0$ and $\lambda \geq 0$. Let $f$ be a positive $\lambda$-harmonic function in $C(\overline{B(x,r+l)})$. There exists a constant $r_* \in [r,r+l]$ such that
\begin{equation}\label{harnack}
\displaystyle \int_{B(x.r)} |(\log f(y))'|^2\,d\mu (y)\leq 2|\partial B(x,r_*)|/l.
\end{equation}
\end{prop}
\begin{proof}Set $F(s) =\int_{B(x,s)} |(\log f)'|^2 d\mu$. Let $\{ r=r_0 < r_1<\cdots <r_n=r+l \}$ be the union of the set $\{r_0=r, r_n=r+l\}$ and the set of all radii such that a branching appears at a point in the sphere $S(x, r_i)$, for $i=1, \cdots n-1$ .   \\
\textit{\underline{Step 1. Computation on non-branching parts:}}  For any $i$, we find a lower bound of $\frac{F'(s)}{F(s)^2}$ on $(r_i,r_{r+1}].$ Let us first compute $F(s)$ on each interval $(r_i,r_{r+1}]$ (i.e. where there is no branching). By Lemma \ref{bwws2}, $f \in D^\infty(\X)$ and we have 
\begin{equation}\label{par}
\Delta \log f = (\log f)''=  f''/f -|(\log f)'|^2=-\lambda-|(\log f)'|^2.
\end{equation}
For any $s$ and $\delta$ such that $(s-\delta,s] \subset (r_i,r_{i+1})$, choose a non-negative function $\varphi$ in $C^\infty(B(x,r+l))$ satisfying 
\begin{displaymath}
\varphi|_{B(x,s-\delta)} \equiv 1 \text{  and  } \varphi|_{B(x,s)^c} \equiv 0 \text{ and } |\varphi'| \leq \frac{2}{\delta}.
\end{displaymath}  
Since $f$ is continuous and satisfies Kirchhoff law, i.e. $f'$ satisfies \eqref{Kirchhoff}, it follows that $(\log f)'$ also satisfies \eqref{Kirchhoff} when $k=0$. By the integration by part and the equation \eqref{par}, we obtain 
\begin{equation}\label{h}
\begin{split}
&\int_{B(x,s)\backslash B(x,s-\delta)} \varphi'(\log f)'d\mu=\int_{B(x,s)} \varphi'(\log f)'d\mu\\
&=\sum_{v\in B(x,s)} \left\{\sum_{\substack{e\in E\\t(e)=v}} \varphi(v) (\log f|_e)'(v)-\sum_{\substack{e\in E\\ v=i(e)}}\varphi(v) (\log f|_e)'(v)\right\}-\int_{B(x,s)} \varphi(\log f)''d\mu\\
&\overset{\substack{\eqref{Kirchhoff}+\\ \eqref{par}}}=\displaystyle \int_{B(x,s)}\lambda\varphi d\mu +\int_{B(x,s) }\varphi|(\log f)'|^2d\mu\geq\int_{B(x,s) }\varphi|(\log f)'|^2d\mu.
\end{split}
\end{equation}
Note that the first equality follows from the support of $\varphi'$.
 Let $r_{i,*}$ be the midpoint of the interval $[r_{i},r_{i+1}]$. 
By H\"older's inequality and $\eqref{h}$,
\begin{equation}
\begin{split}
\nonumber\displaystyle&\int_{B(x,s)} \varphi |(\log f)'|^2 d\mu \overset{\eqref{h}}\leq\int_{B(x,s)} \varphi'(\log f)'d\mu \\ 
&\leq  \displaystyle \biggl(\frac{4}{\delta}\int_{B(x,s)\backslash B(x,s-\delta)}|(\log f)'|^2d\mu\biggr)^{\frac{1}{2}}(|\partial B(x,r_{i,*})|)^{\frac{1}{2}}.
\end{split}
\end{equation}
Since $B(x,r_{i+1}) \backslash B(x,r_i)$ is a disjoint union of $|\partial B(x,r_{i,*}))|$ distinct intervals of length $r_{i+1}-r_i$, we have
\begin{displaymath}
\frac{1}{|\partial B(x,r_{i,*})|} \leq \frac{4F'(s)}{F(s)^2}.
\end{displaymath}

Integrating over the interval $(r_i,r_{i+1}]$,
\begin{equation}\label{harnack2}
\displaystyle \frac{r_{i+1}-r_i}{|\partial B(x,r_{i,*})|}=\int_{r_i}^{r_{i+1}} \frac{1}{|\partial B(x,r_{i,*})|} \leq \displaystyle \int_{r_i}^{r_{i+1}} \frac{4F'(s)}{F(s)^2}ds=\frac{4}{F(r_i)}-\frac{4}{F(r_{i+1})}. \\
\end{equation}
\textit{\underline{Step 2. Completing the proof:}} Using the inequality \eqref{harnack2}, we have
$$\displaystyle\sum_{i=0}^{n-1}  \frac{r_{i+1}-r_i}{|\partial B(x,r_{i,*})|} \leq \displaystyle\sum_{i=0}^{n-1}\left(\frac{4}{F(r_i)}-\frac{4}{F(r_{i+1})}\right)\leq\frac{4}{F(r)}.$$ 
Let $r_*$ be a number such that $|\partial B(x,r_*)| = \underset{0\leq i \leq n-1}\max |\partial B(x,r_{i,*})|.$
We obtain 
\begin{displaymath}
\frac {l}{|\partial B(x,r_*)|} =\frac{\sum_{i=0}^{n-1} r_{i+1}-r_{i}}{|\partial B(x,r_*)|}\leq \displaystyle \sum_{i=0}^{n-1}\frac{r_{i+1}-r_{i}}{|\partial B(x,r_{i,*})|} \leq \frac{4}{F(r)},
\end{displaymath}
thus the inequality \eqref{harnack} holds.
\end{proof}

Using Proposition \ref{gphar}, we obtain the graph version of Harnack inequality.

\begin{coro}\label{G2} Fix $\lambda>0$. Let $f$ be a positive $\lambda$-harmonic function in $C(\overline{B(x,r+l)})$. Then there exists an explicit constant $D_{r,l}=(4|\partial B(x,r_*)|\mu(B(x,r)/l)^{\frac{1}{2}}$ such that for all $y,z \in B(x,r)$,
\begin{equation}\label{harnack3} 
\left|\log\frac{f(y)}{f(z)}\right| \leq  D_{r,l}.
\end{equation}
\end{coro}
\begin{proof}By H\"older inequality,
\begin{align*}
\displaystyle\left|\log \frac{f(y)}{f(z)}\right|&=\displaystyle\left|\int_{[y,z]}(\log f)'d\mu\right|\leq  \int_{B(x,r)}|(\log f)'|d\mu\\
&\leq \displaystyle\left\{\mu(B(x,r))\int_{B(x,r)}|\log f'|^2d\mu \right\}^\frac{1}{2}\leq\sqrt{4|\partial B(x,r_*)|\mu(B(x,r))/l}.
\end{align*}
\end{proof}
By compactness of $X$, we can choose $D_{r,l}$ satisfying \eqref{harnack3} indepedent of $x$.
\subsection{Existence of $\lambda$-harmonic functions}\label{sec:2.5} 
In this section, for the bottom of the $L^2$-spectrum $\lambda_0$ of $-\Delta$, we prove that for all $\lambda\in [0,\lambda_0]$, a positive $\lambda$-harmonic function exists. The existence of a positive $\lambda_0$-harmonic function will be used to prove Theorem \ref{2.32}. 
 \begin{defn} 
 \emph{The bottom of the spectrum $\lambda_0$ of $-\Delta$} is defined as follows: $$\lambda_0:=\inf \left\{\frac{||f'||_{L^2( \widetilde{X})}^2}{||f||_{L^2( \widetilde{X})}^2}:f \in W_0^1( \widetilde{X})\backslash\{0\}\right\}.$$
 \end{defn}
 We will see in Proposition \ref{spect} that $\lambda_0$ is indeed the infimum of the spectrum of $-\Delta$. Since $\G$ is non-amenable, the bottom of the spectrum is non-zero (\cite{SW} Theorem 8.5).

 Let $\cE$ be an arbitrary strongly local, regular Dirichlet form. If the heat kernel is positive (which is our case by Theorem~\ref{positive}), the existence of a positive $\lambda$-harmonic function for $\lambda \in [ 0, \lambda_0]$ is proved under the condition of the local compactness property and Harnack principle \cite{LSV}. Our space $\X$ satisfies the local compactness property and a modified version of uniform Harnack principle (Lemma~\ref{lem:3.5}), which suffice to use the result of \cite{LSV}. Now let $\mathcal{E}$ be the Dirichlet form defined in Definition \ref{def:2.2}.
\begin{lem}\label{lem:3.5.1} The Dirichlet form $\cE$ satisfies the local compactness property, i.e. for every precompact open $O \subset \ X$,
$W_0^1(O)$ is compactly embedded in $L^2(\X).$
\end{lem}
 \begin{proof} The lemma follows from a graph version of Rellich theorem on compactness of the embedding of Sobolev spaces in $L^2$.
 \end{proof}
 \begin{lem} \label{lem:3.5}
 The Dirichlet form $\cE$ satisfies uniform Harnack principle, i.e for every bounded interval $I\subset \mathbb{R}_{\geq 0}$, for every precompact connected open subset $O$ of $\X$ and for every sequence $\{f_n\}$ of positive $\lambda_n$-harmonic functions on $O$ with $\lambda_n\in I$, the following holds: If for some compact set $K\subset O$ with $d(K,\partial O)\geq 1$, 
 \begin{equation}\label{eq:3.7}
 \sup_{n\in \mathbb{N}}||f_n1_K||_{L^2(O)}<\infty,
 \end{equation}
 then for any compact $K'\subset O$ with $d(K',\partial O)\geq 1$, we have
 \begin{equation}\label{eq:3.8}
 \sup_{n\in \mathbb{N}}|| f_n 1_{K'}||_{L^2(O)}<\infty.
 \end{equation}
 \end{lem}
  \begin{proof}By Corollary \ref{G2}, for any points $x,y\in\{z\in O:d(z,\partial O)\geq 1\}$ and $n\in \mathbb{N}$, we have $f_n(x)\leq e^{diam(O)D_{1/2,1/2}}f_n(y)$. For any $x\in K$ and $y\in \{z\in O:d(z,\partial O)\geq 1\},$ we have
  \begin{equation}
  \begin{split}
 \nonumber (\int_K f_n^2d\mu)^{1/2}&\geq e^{-diam(O)D_{1/2,1/2}}(\mu(K))^{1/2} f_n(x)\\
 \nonumber&\geq e^{-2diam(O)D_{1/2,1/2}}(\mu(K))^{1/2} f_n(y).
 \end{split}
 \end{equation}
  This shows that for any compact set $K'$ with $d(K',\partial O)\geq 1$ and $y\in K'$,
  \begin{equation}
  \begin{split}
\nonumber  ||f_n 1_{K'}||_{L^2(O)}&\leq e^{diam(O)D_{1/2,1/2}}{\mu(O)}^{1/2}f_n(y)\\\nonumber&\leq e^{3diam(O)D_{1/2,1/2}}\biggl(\frac{\mu(O)}{\mu(K)}\biggr)^{1/2} \sup_{n\in \mathbb{N}}||f_n1_K||_{L^2(O)}.
\end{split}
 \end{equation}
  Uniform Harnack principle follows from the above inequality.
   \end{proof} 

\begin{lem}\label{lem:3.9} Let $\lambda \in (-\infty, \lambda_0).$ The inverse operator $ (-\Delta-\lambda)^{-1}$ is well-defined and for any compactly supported smooth function $g$,
 $$(-\Delta-\lambda)^{-1} g =\int_{\X}  \int_0^\infty e^{\lambda t} p(t,x,y) g(y) d\mu(y).$$
\end{lem}
\begin{proof}
For all $f\in Dom(\Delta)$ and for all $\lambda\in (-\infty,\lambda_0)$,
$$\left\langle-\Delta f,f\right\rangle-\lambda\left\langle f,f\right\rangle=\mathcal{E}(f,f)-\lambda\left\langle f,f\right\rangle\geq(\lambda_0-\lambda)\left\langle f,f\right\rangle.$$
The dimension of the kernel of $(-\Delta-\lambda I)$ is zero. By Proposition 1.6 in \cite{Sc}, $(-\Delta-\lambda I)$ is surjective.
By Proposition 2.1 in \cite{Sc}, $(-\Delta -\lambda)$ has a bounded inverse operator. The second statement is a direct calculation. \end{proof}   
   
\begin{thm}\label{harmonic}\emph{(Modified version of Theorem 3.7. of \cite{LSV})}  Lemma \ref{lem:3.5.1} and Lemma~\ref{lem:3.5} imply the existence of  a positive $\lambda$-harmonic function for any $\lambda \in [0, \lambda_0].$
\end{thm}
\begin{proof} 
The proof of \cite{LSV} is as follows. We first choose an exhausting sequence of increasing precompact connected open subsets $\{O_m\}_{m\geq 1}$, i.e. for any $m$, $\overline{O_m}$ is a subset of $O_{m+1}$ and $\bigcup O_m=\X$. Choose a sequence $\{\lambda_n\}_{n\geq1}\subset [0,\lambda_0)$ with $\underset{n\rightarrow\infty}\lim \lambda_n=\lambda$ and a sequence of compactly supported non-trivial functions $g_n \geq 0$ with $\text{supp }g_n\subset \overline{O_{n+2}}^c.$ Denote $f_n=(-\Delta-\lambda_n)^{-1}g_n$. By the positivity of the heat kernel and Lemma~\ref{lem:3.9}, for any $n$, $f_n$ is positive almost everywhere. By definition, the function $f_n$ is a $\lambda_n$-harmonic function on $\X\backslash \text{supp }g_n$.  Using uniform Harnack principle, $\underset{n\in \mathbb{N}}\sup||f_n1_{\overline{O}_{m}}||_2<\infty$ when $||f_n1_{\overline{O}_{1}}||_2=1$ for any $n$. Using the diagonal argument, it follows that there is a weak limit $f$ of $f_n$ and $f$ is a positive $\lambda$-harmonic function. In our modified version, we only need to choose an exhausting sequence of open sets $O_m$ with the additional condition $d(\overline{O_m},\partial O_{m+1})\geq 1$, which is satisfied by simply taking $O_n=B(x,n)$.
\end{proof}

\subsection{Green functions of graphs}\label{3}
In this section, we first define the $\lambda$-Green function and Green region $\lambda \in (-\infty, \lambda_0].$
Using the existence of a positive $\lambda$-harmonic function proved in Section \ref{sec:2.5}, we prove that the $\lambda$-Green function on $\X$ is finite for all $\lambda\in(-\infty,\lambda_0]$. 
\begin{defn}
\begin{enumerate}
 \item[(1)] The $\lambda$-\emph{Green function} is defined as follows:
 $$G_\lambda(x,y):=\int_0^\infty e^{\lambda t}p(t,x,y)dt.$$
 The \emph{Green region} is the set of $\lambda \in \mathbb{R}$ for which $G_{\lambda}$ is finite.
\item[(2)]  The resolvent set of $-\Delta$ on $Dom(\Delta)$ is the set of $\lambda\in \mathbb{C}$ such that $(-\Delta-\lambda I)$ has a bounded inverse operator on $\mathcal H$. The spectrum of $-\Delta$ is the complement of the resolvent set. 
\end{enumerate}
\end{defn}
   \begin{prop}\label{spect} The bottom of the spectrum $\lambda_0$ is in the closure of the spectrum of the operator $-\Delta$.
\end{prop}
\begin{proof}By definition, for any $\lambda\in (\lambda_0,\infty)$, there exists a function $f$ in $C_c^\infty(\X)$ such that  $$\lambda_0<\frac{||f'||^2_{L^2(\X)}}{||f||^2_{L^2(\X)}}< \lambda.$$
Let $O$ be a precompact open set containing the support of $f$. By Theorem \ref{obasis}, the eigenvalue of the eigenfunction $p^{O}_{1}$ (in the equation \eqref{basis}) of $-\Delta$, which is the bottom of the spectrum on $O$, is smaller than $\lambda.$ Since $-\Delta$ is self-adjoint, the spectrum of $-\Delta$ is in $\mathbb{R}$. By Lemma~\ref{lem:3.9}, the closure of the spectrum of $-\Delta$ is contained in $[\lambda_0,\infty).$\end{proof}
As in the page 338 of \cite{Su}, if $\lambda$ is in the Green region, for any compact set $K$ and $x\in \X$, $\displaystyle\lim_{t\rightarrow\infty} P^\lambda_t 1_{K}(x)=0$ where $1_K$ is the characteristic function of $K$. For any $\lambda$ with $\lambda>\lambda_0$, there exists a connected precompact open set $O$ in $\X$ such that $\lambda>\lambda_o^O$. Denote by $p_O(t,x,y)$ the heat kernel on $O$ defined by \eqref{A.2}. Since $\displaystyle \lim_{t\rightarrow \infty} e^{\lambda^{O}_{1} t}p_{O}(t,x,y)=p^{O}_1(x)p^{O}_1(y)$, $\underset{t\rightarrow\infty} \lim P^\lambda_t 1_{\overline{O}}(x)=\infty$, thus one obtains the following corollary. 
\begin{coro}\label{g} If $\lambda$ is an element of the resolvent set of $-\Delta$, then the inverse operator of $(-\Delta-\lambda I)$ is described by the following integral: for all $f\in L^2( \widetilde{X})$,
$$(-\Delta-\lambda I)^{-1}(f)(x)=\displaystyle \int_{ \widetilde{X}} G_{\lambda}(x,y)f(y)d\mu.$$
Furthermore, the Green region is $(-\infty,\lambda_0)$ or $(-\infty,\lambda_0]$ .
\end{coro} 

As in the page 340 in \cite{Su}, if the set of positive $\lambda_0$-harmonic functions $f$ satisfying $f(x_0)=1$ for a fixed $x_0$ has more than one element, the $\lambda_0$-Green function is finite. 
\begin{lem}\label{conv}If there exists a positive $\lambda$-superharmonic function $f$, i.e. f is a increasing limit of continuous functions and for any $x\in\X$ and $t>0$, $$ P^{\lambda}_tf(x):=\int_{\X}e^{\lambda t} p(t,x,y)f(y)d\mu(y)\leq f(x),$$ and $f$ is not $\lambda$-harmonic, then the $\lambda$-Green function is finite. 
\end{lem}
\begin{proof}Suppose that the $\lambda$-Green function diverges. If $f$ is $\lambda$-harmonic, then $-\Delta f(x)$ is equal to $\lambda f(x)$ for $\mu$-a.e. $x$. By the definition of the heat kernel,
$$f(x)=e^{t(\Delta+\lambda)}f(x)=P^\lambda_tf(x)\text{ for }\mu\text{-a.e }x.$$ Thus, for any positive $\lambda$-superharmonic function $f$ which is not $\lambda$-harmonic and for any $t>0$, there exists $\epsilon>0$ such that $\mu(\{x\in \X:f(x)-P^\lambda_tf(x)>\epsilon\})>0.$ Choose a measurable set $O\subset\{x\in \X:f(x)-P^\lambda_tf(x)>\epsilon\}$ with $0<\mu(O)<\infty.$ For any sufficiently large $T>t$, we have
 \begin{equation}
 \begin{split}
 \nonumber&\int_0^{T}\int_O e^{\lambda s}p(s,x,y)\biggl(\frac{f(y)-P_t^\lambda f(y)}{t}\biggr)d\mu(y)ds\leq  \int_0^{T} P_s^\lambda\biggl(\frac{f(x)-P_t^\lambda f(x)}{t}\biggr)ds\\
 \nonumber&=\frac{1}{t}\int_0^{T}P^\lambda_{s}f(x)-P^\lambda_{s+t}f(x)ds=\frac{1}{t}\int_0^{t}P^\lambda_{s}f(x)ds-\frac{1}{t}\int_T^{T+t}P^\lambda_{s}f(x)ds\\ 
 \nonumber&\leq \frac{1}{t}\int_0^{t}f(x)ds-\frac{1}{t}\int_T^{T+t}P^\lambda_{T+t}f(x)ds=f(x)-P^\lambda_{T+t}f(x)\leq f(x).
 \end{split}
 \end{equation}
 The above inequality shows that $\frac{\epsilon}{t} \int_OG_\lambda(x,y)d\mu(y)\leq f(x).$ This contradicts that the $\lambda$-Green function diverges. Therefore, the $\lambda$-Green function is finite.
\end{proof}
\begin{prop}\label{lamhar}\emph{(\cite{Su})} If the Green region is $(-\infty,\lambda_0)$, then there exists a unique positive $\lambda_0$-harmonic function up to a constant multiple. 
\end{prop}
\begin{proof}Fix a point $x_0$. Let $\cH_{\lambda_0}$ be the set of positive $\lambda_0$-harmonic functions $f$ on $\X$ satisfying $f(x_0)=1.$ Suppose that $\cH_{\lambda_0}$ has more than one elements. Corollary \ref{G2} shows that the functions in $\cH_{\lambda_0}$ are equicontinuous and $\{f(x):f\in\cH_{\lambda_0}\}$ has a compact closure for each $x\in \X$. By Arzela-Ascoli's theorem (See \cite{Mu} Theorem 47.1), $\cH_{\lambda_0}$ is a convex compact set with respect to the topology of uniform convergence on compact sets. We can choose two distinct extreme points $f_1$ and $f_2$ of $\cH_{\lambda_0}$ satisfying neither $f_1\geq f_2$ nor $f_1\leq f_2$. 

Put $f(x)=\min\{f_1(x),f_2(x)\}$. Then $P^{\lambda_0}_tf(x)<P^{\lambda_0}_tf_i(x)=f_i(x)$ for all $i=1,2$ and $x\in \X$. The function $f$ is a $\lambda_0$-superhamonic function. By Lemma \ref{conv}, the $\lambda_0$-Green function is finite.
\end{proof}
Following the proof of Theorem 3 in \cite{LS}, we show the following theorem.
\begin{thm}\label{2.32}If $\Gamma$ is non-amenable, then the $\lambda_0$-Green function is finite.
\end{thm}
\begin{proof} By Theorem \ref{harmonic}, there exists a positive $\lambda_0$-harmonic function $f$. Denote $q(t,x,y):=e^{\lambda_0 t}p(t,x,y)\frac{f(y)}{f(x)}.$ Define a semigroup $\{Q_t\}$ as follows: for any $g$ in $L^\infty(\X),$
$$Q_tg(x):=\int_{\X}q(t,x,y)g(y)d\mu(y).$$
By the continuity of the heat kernel and $f$, for any $x$ and $t,t_0>0$,
\begin{equation}\label{...}
\underset{x'\rightarrow x}\lim \int_{\X}\left|q(t,x,y)-q(t,x',y)\right|=0\text{ and }
\end{equation}
  \begin{equation}\label{....}
   \underset{s\rightarrow t}\lim \int_{\X}\left|q(s+t_0,x,y)-q(t+t_0,x,y)\right|=0.
\end{equation}
 Since $||Q_tg||_{L^\infty(\X)}\leq ||g||_{L^\infty(\X)}$ for all $g\in L^\infty(\X)$ and $t>0$, the equation \eqref{....} implies that for any fixed $x\in\X$, $t_0>0$ and $g\in L^\infty(\X)$, $g_x(t):=Q_{t+t_0}g(x)$ is a bounded continuous function on $\mathbb{R}_+$.

 Applying that $Q_{t+s}=Q_sQ_t$ and $Q_t 1=1$, for any $g\in L^\infty(\X)$ and $x,x'\in \X$, we have
\begin{equation}\label{.....}
\begin{split}
&\sup_{t>0} |Q_{t+t_0}g(x)-Q_{t+t_0}g(x')|\\&= \sup_{t>0} \left|\int_{\X}q(t+t_0,x,y)g(y)-q(t+t_0,x',y)g(y)d\mu(y)\right|\\
&\leq \sup_{t>0} \int_{\X}\int_{\X}|q(t_0,x,z)q(t_0,z,y)g(y)-q(t_0,x',z)q(t,z,y)g(y)|d\mu(y)d\mu(z)\\
&\leq \sup_{t>0} \int_{\X}|q(t_0,x,z)-q(t_0,x',z)|q(t,y,z)|g(y)|d\mu(y)d\mu(z)\\
&= \sup_{t>0} \int_{\X}|q(t_0,x,z)-q(t_0,x',z)|Q_t|g|(z)d\mu(z)\\
 &\leq||g||_{L^\infty(\X)} \int_{\X}\left|q(t_0,x,z)-q(t_0,x',z)\right|d\mu(z).
 \end{split}
\end{equation}
 The equation \eqref{...} and the inequality \eqref{.....} show that for any fixed $g\in L^\infty(\X)$ and $t_0>0$, the map $x\mapsto g_x(t)$ is a continuous map from $\X$ to $L^\infty(\mathbb{R}_+)$ with respect to $||\cdot||_{L^\infty(\mathbb{R}_+)}$-norm.
  
Denote $u_{s}(t):=u(t+s).$ Since $\mathbb{R}_+$ is amenable, there exists a linear functional $\varphi$ on $C_b(\R_+)$ such that for all $s\in \mathbb{R}_+$ and for all $u \in L^\infty(\R_+),$ $\varphi(u_{s})=\varphi(u)$ and $\|\varphi\|_{op}\leq1$.

 The map $\widetilde{\varphi}$ from $L^\infty(\X)$ to the subspace $C_b(\X)$ is defined as follows: for any $g\in L^\infty(\X),$ $$\widetilde{\varphi}(g)(x):=\varphi(Q_{t+t_0}g(x)).$$   
For any partition $\cP$ of $\X$ and for any $y_p\in P\in \cP$, we have
 \begin{equation}
 \nonumber\sum_{P\in\cP}\mu(P)q(s,x,y_p)\varphi(Q_{t+t_0}g(y_p))= \varphi\biggl(\sum_{P\in\cP}\mu(P)q(s,x,{y_p})Q_{t+t_0}g(y_p)\biggr).
 \end{equation}
As the diameter of $\cP$ goes to zero, we have $Q_s\widetilde{\varphi}(g)(x)=\varphi(Q_{t+t_0+s}g(x))$.  Since for any fixed $x$, $\varphi(g_x(t+s))=\varphi(g_x(t))$, we obtain $Q_s\varphi(g_x(t))=\varphi(g_x(t)).$ 
 
 Let $\mathcal{H}_{\lambda_0}$ be the set of functions $g\in C_b(\X)$ satisfying $Q_tg=g.$ Obviously, the image of $\widetilde{\varphi}$ is in $\cH_{\lambda_0}$.
 By the definition of $Q_t$, for any $g\in \mathcal{H}_{\lambda_0}$, $P_t^\lambda (f g_1)(x)=f(x)g_1(x),$ where $g_1=g+||g||_{L^\infty(\X)}+1.$ The operator $\Delta+\lambda I$ is the limit $\lim_{t\rightarrow 0+} \frac{P_t^\lambda -I}{t}$ in the sense of distributions, i.e. for any $f_1\in W_c^1(\X)$, \begin{equation}\label{3.10}\lim_{t\rightarrow0+}\frac{1}{t}({P_t^\lambda fg_1-fg_1},f_1)=-\mathcal{E}(fg_1,f_1)-\lambda(fg_1,f_1).\end{equation} 
The left side of \eqref{3.10} is zero when $g\in \cH_{\lambda_0}$. Thus the function $fg_1$ is $\lambda$-harmonic. 

Suppose that the $\lambda_0$-Green function diverges. By Proposition \ref{lamhar}, $fg_1$ is a constant multiple of $f$. Hence, $g_1$ and $g$ are constant functions. The space $\mathcal{H}_{\lambda_0}$ consists of constant functions and the map $\widetilde{\varphi}$ is a linear functional on $L^\infty(\X)$. Any function $g$ in $L^\infty(\G)$ is regarded as a function in $L^\infty(\X)$, by defining as follows $g(x)=g(\gamma)$, where $x\in \gamma F$. 
As in \cite{LS}, the functional $\widetilde{\varphi}$ on $L^\infty(\Gamma)$ is $\G$-invariant and $||\widetilde{\varphi}||_{op}\leq 1$. Hence, the group $\Gamma$ is amenable.
\end{proof}
The next lemma follows from Corollary \ref{g} and Corollary \ref{G2}.
\begin{coro}\label{harnack5}For any $\lambda\in [0,\lambda_0]$, the map $y\mapsto G_\lambda(x,y)$ defined on $\X\backslash \{x\}$ is $\lambda$-harmonic. The restriction of  the map $y\mapsto G_\lambda(x,y)$ on $B(z,r+l)\subset\X\backslash \{x\}$ satisfies Harnack inequality, i.e. for any $z_1,z_2\in B(z,r)$,
\begin{equation}\label{harnack4}
{G_\lambda(x,z_1)}\leq e^{D_{r,l}}{G_\lambda(x,z_2)},
\end{equation}
where the constant $D_{r,l}$ is from Corollary \ref{G2}.
\end{coro}
\begin{proof} By Corollary \ref{g}, $\underset{t\rightarrow 0+}\lim\frac{I-P^\lambda_t}{t}G_\lambda(x,y)=\delta_{x-y}$ in the sense of distributions. Hence, the map $y\mapsto G_\lambda(x,y)$ is $\lambda$-harmonic function on $\X\backslash \{x\}$. \end{proof}
Since the $\lambda$-Green function $G_\lambda(x,y)$ is $\lambda$-harmonic on $\widetilde{X}\backslash \{x\}$, Harnack inequality  will appear the proof of Ancona-Gou\"ezel inequality \eqref{1.1.0}.
\section{Martin boundary of hyperbolic graphs}\label{4}
 For any Riemannian manifold with negatively pinched curvature, Ancona showed that there exists a constant $C_\epsilon$ satisfying the inequality \eqref{1.1.0} when $\lambda \in [0,\lambda_0-\epsilon)$ (\cite{A}). Using the inequality \eqref{1.1.0}, Ancona proved that the Gromov boundary of a hyperbolic group coincides with the $\lambda$-Martin boundary for all $\lambda\in[0,\lambda_0)$. Using the ideas of \cite{GL} and \cite{G}, Ledrappier and Lim showed that the Gromov boundary of the universal cover of a negatively curved closed manifold coincides with its $\lambda_0$-Martin boundary (\cite{LL}). We extend their results to Gromov hyperbolic graphs.
\subsection{Brownian motion} 
For a general regular Dirichlet form, one can associate a strong Markov process on the set $\Omega$ of continuous paths, which is so-called a Hunt process. In the context of metric measure spaces, this Hunt process is often called Brownian motion (\cite{BSSW}, \cite{BK}, \cite{PS}, \cite{St4} for example). Let us recall this process and its basic properties (see Section 7.3. of \cite{FOT}).
Let $\cB$ be the Borel $\sigma$-algebra of the metric space $\X.$

  \begin{defn}\label{def:4.2} Let $\Omega_x$ be the set of continuous paths in $\X$ starting at $x$ and $\omega \in\Omega_x$. \begin{enumerate}
\item[(1)]The family $\mathbb{P}=\{\mathbb{P}_x\}_{x\in \X}$ of probability measures indexed by $\X$ is defined as follows: for any $0<t_1<\cdots<t_n$ and any Borel sets $B_n$ in $ \widetilde{X}$,
\begin{align*}
\displaystyle &\mathbb{P}_x[\omega({t_1})\in B_1,\cdots, \omega({t_n})\in B_n]\\
&=\displaystyle \int_{B_1\times \cdots \times B_n} p(t_1,x,y_1)p(s_1,y_1,y_2)\cdots p(s_{n-1},y_{n-1},y_n)d\mu^n(y_1,\cdots,y_n),
\end{align*}
 where $\mu^n=\mu\times\cdots\times \mu$ and $s_i=t_{i+1}-t_{i}$.  The expectation of a function $f$ on $\Omega_x$ is defined by
$\mathbb{E}_x(f)=\int_{\Omega_x}f\mathbb{P}_x.$ 
\item[(2)] For a continuous path $\omega$ in $ \widetilde{X}$, the exit time $\sigma_O$ of $O$ is defined by
$$\sigma_O(\omega)=\inf\{t\geq 0: \omega(t) \in  O^c\}.$$

\end{enumerate}
\end{defn}
Define $\Omega$ be the set of all continuous paths in $\X$. Let $\cM_t$ be the $\sigma$-algebra generated by the sets $\{\omega\in \Omega:\omega(s)\in B\}_{B\in \cB,s\leq t}$.  Denote $\cM=\{\cM_t\}_{t\geq0}$. Define a family of measurable maps $Y_t:\Omega\rightarrow \X$ by $Y_t(\omega)=\omega(t).$
A function $\sigma$ on $\Omega$ is a \textit{stopping time} if $$\{\omega \in \Omega:\sigma(\omega)\leq t\}\in \mathcal{M}_t$$ for all $t$. For a stopping time $\sigma$, denote $$\mathcal{M}_\sigma:=\{A\in \mathcal{M}:A\cap \{\omega\in\Omega:\sigma(\omega)\leq t\}\in \mathcal{M}_t, \forall t\geq 0\}.$$ For convenience, denote $\omega(\sigma):=\omega(\sigma(\omega))$, $e^{\lambda \sigma}:=e^{\lambda \sigma(\omega)}$ and $1_M:=1_M(\omega)$ for all $\omega\in \Omega$ and $M\in \cM$.
 \begin{prop} The process $(\Omega, \mathcal{M}, \{Y_t\}_{t\geq0}, \mathbb{P})$ is a strong Markov process with state space $(\X,\mathcal{B})$, i.e. a Markov process with the following property: For any $x\in \X$, $t\geq0$, $B\in \mathcal{B}$ and any stopping time $\sigma$, 
 \begin{equation}\label{4.1}\mathbb{E}_x[1_{\sigma<\infty} 1_B(\omega(\sigma +t))|\mathcal{M}_\sigma]=1_{\sigma<\infty}\mathbb{E}_{\omega(\sigma)}[1_B(\omega(t))]
\end{equation}
and $\cM_t=\mathcal{M}_{t+}:=\underset{t'>t}\bigcap \mathcal{M}_{t'}$. \end{prop}

\subsection{Proof of Ancona-Gou\"ezel inequality}\label{sec:3.2}
Using the properties of a strong Markov process as in \cite{LL}, we prove Ancona-Gou\"ezel inequality. In this section, the constant $C$ may vary from line to line. 

  The probability measure $\mathbb{P}_x$ describes the Brownian motion related to our Laplacian on $\X$ (see \cite{BSSW} and \cite{FOT}). The process derived from a strongly local regular Dirichlet form is a strong Markov process on continuous paths with probability 1 (\cite{FOT} Theorem 4.5.3 and Theorem 7.3.1).  

  The exit time $\sigma_O$ of a connected open set $O$ in Definition \ref{def:4.2}  is an example of the stopping time.
 Using the exit time $\sigma_O$ of an open set $O$, the relative Green function $G_\lambda(x,y:O)$ is defined as follows: for any point $x \in O$ and for any Borel measurable function $f$ on $O$,
$$\int_OG_\lambda(x,y:O)f(y)d\mu=\displaystyle\mathbb{E}_x\biggl[\int_0^{\sigma_O}e^{\lambda t}f(\omega(t))dt\biggr].$$

By strong Markov property, we have the following proposition as in \cite{LL}.
\begin{prop}\label{MArkov} Let $O_1$ and $O_2$ be connected open sets in $ \widetilde{X}$ intersecting each other. Then we have the following: for all $x\in O_1\backslash\overline{O_2}$ and $y\in O_1\backslash{\partial O_2}$ and for all $\lambda\in[0,\lambda_0]$,
\begin{equation}\label{exit}
G_\lambda(x,y:O_1)=\mathbb{E}_x[1_{\tau< \sigma}e^{\lambda\tau}G_\lambda(\omega(\tau),y:O_1)]+G_\lambda(x,y:O_1\backslash \overline{O_2}),
\end{equation}
where $\sigma=\sigma_{O_1}$ and $\tau=\sigma_{O_1\backslash \overline{O_2}}$ are the exit times of $O_1$ and $O_1\backslash \overline{O_2}$, respectively.
\end{prop}
\begin{proof}The proof of \cite{LL} is as follows. Choose a ball $B(y,r)\subset O_1\backslash{\partial O_2}$. Using strong Markov property, one obtains
\begin{equation}\label{4.2.11}
\begin{split}&\int_{B(y,r)}G_\lambda(x,z:O_1)d\mu(z)\\ \nonumber
&=\mathbb{E}_x[1_{\tau< \sigma}e^{\lambda\tau}\int_{B(y,r)}G_\lambda(\omega(\tau),z:O_1)d\mu(z)]+\int_{B(y,r)}G_\lambda(x,z:O_1\backslash \overline{O_2})d\mu(z).\nonumber
\end{split}
\end{equation}
To complete the proof, we need Lebesgue's theorem, which is clear if a point $x$ is in an open edge $e^o$. 
Let us verify Lebesgue theorem for vertices of $ \widetilde{X}$. For a vertex $x$ in $ \widetilde{X}$ and a continuous function $f$ on $B(x,r)$, 
$$\displaystyle\frac{1}{B(x,r)}\int_{B(x,r)}fd\mu=\displaystyle\frac{1}{\deg(x)}\biggl(\sum_{i(e)=x}\frac{1}{r}\int_0^{r} f|_edt+\sum_{t(e)=x}\frac{1}{r}\int_{l_e-r}^{l_e}f|_edt\biggr).$$
For any edge $e$ with $i(e)=x$ ($t(e)=x$, resp), $$\lim_{r\rightarrow0+}\frac{1}{r}\int_0^rf|_edt=f(x) \left(\lim_{r\rightarrow0+}\frac{1}{r}\int_{l_e-r}^{l_e}f|_edt=f(x),\text{ resp}\right).$$
Using above equation, we obtain the graph version of Lebesgue's theorem as follows:
$$\displaystyle\lim_{r\rightarrow 0+}\frac{1}{B(x,r)}\int_{B(x,r)}fd\mu=\displaystyle\lim_{r\rightarrow 0+}\frac{1}{r\deg(x)}\biggl(\sum_{i(e)=x}\int_0^{r} fdt+\sum_{t(e)=x}\int_{l_e-r}^{l_e}fdt\biggr)=f(x).$$
\end{proof}
\begin{lem}\label{AC!!} 
For any given $r>0$, there exists a constant $C_r$ such that for all $\lambda \in [0,\lambda_0]$ and for all $(y,z)$ in $\overline{B(x,r)}^2$ with $d(y,z)\geq r/2$, $$G_\lambda(y,z: B(x,2r))\geq C_r.$$
\end{lem}
\begin{proof}Suppose on the contrary that there exists a sequence of points $(\lambda_n,y_n,z_n)$ in $[0,\lambda]\times \overline{B(x,r)}^2$ such that $$G_{\lambda_n}(y_n,z_n:B(x,2r))\leq \frac{1}{n}.$$
Since $[0, \lambda]\times\overline{B(x,r)}^2$ is compact, there is a subsequence $\{(\lambda_{n_k}, y_{n_k},z_{n_k})\}_{k=1}^\infty$ converging to, say $(\lambda, y,z)$.     
By continuity of $G_\lambda$, $G_\lambda(y,z:B(x,2r))=0$. 
Since $$G_\lambda(y,z:B(x,2r))=\displaystyle \int_0^\infty e^{\lambda t}p_{B(x,2r)}(t,x,y)dt=0,$$ we have $p_{B(x,2r)}(t,x,y)=0,$ which contradicts Proposition \ref{Aa} (3). 
\end{proof}
 By compactness of $X$ and Lemma \ref{AC!!}, we remark that there exists a constant $C'_m$ such that for all $\lambda\in[0,\lambda_0]$ and for two distinct points $y,z$ with $1\leq d(y,z)\leq m$,
\begin{equation}\label{eq1.1.2} {C'}_m^{-1}\leq G_\lambda(y,z)\leq C_m'. \end{equation}
However, Harnack inequality \eqref{harnack4} and \eqref{eq1.1.2} together implies only an inequality similar to \eqref{1.1.0} with a constant depending on the distance between points $x$, $y$ and $z$. 

 Let $O_1$ and $O_2$ be connected open sets in $ \widetilde{X}$ intersecting each other.  The measure $\eta_x^{\lambda, O_1\cap \partial O_2}$ on $O_1\cap \partial O_2$ is defined as follows: for any measurable function $f$ on $O_1 \cap \partial O_2,$
 \begin{equation}\label{hmeasure}
\int_{O_1\cap\partial O_2} f(z)d\eta_x^{\lambda,O_1\cap\partial O_2}(z):=\mathbb{E}_x[1_{\tau< \sigma}e^{\lambda\tau}f(\omega(\tau))],
\end{equation}
where $\sigma=\sigma_{O_1}$ and $\tau=\sigma_{O_1\backslash \overline{O_2}}$ are the exit times of $O_1$ and $O_1\backslash \bar{O_2}$, respectively.

Using strong Markov property, we obtain the following lemma, which will be used in the proof of pre-Ancona inequality \eqref{prean}.
\begin{lem}\label{acn} Let $O_1$ and $O_2$ be connected open sets in $ \widetilde{X}$ intersecting each other. Suppose that $O_2$ has the following property: for any $z\in O_1\cap\partial O_2$, \begin{equation}\label{4.3}\mu((B(z,1)\cap O_1\cap O_2)\backslash B(z,1/2))\geq \frac{1}{2}.\end{equation} Let $f$ be a bounded positive function on $\partial O_2$. There exists a constant $C$ such that for all $x\in O_1\backslash\overline{O_2}$ with $d(x,\partial O_2)>2,$
\begin{equation}\label{MARkov}\sum_{z\in O_1\cap \partial O_2}G_\lambda(x,z)f(z)\geq C\int_{O_{1}\cap\partial O_2}f(z)d\eta_x^{\lambda,O_1\cap\partial O_2}.
\end{equation}

\end{lem}
\begin{proof}
 For any point $z\in O_1\cap\partial O_2$, set $A_z:=(B(z,1)\cap O_1\cap  O_2)\backslash B(z,1/2)$. By assumption, $\mu(A_z)\geq 1/2.$ Since the action of $\G$ is cocompact, there exists a constant $C'$ such that $\mu(A_z)<C.$
 By Harnack inequality \eqref{harnack4}, for all $z\in O_1\cap\partial O_2$,
\begin{equation}\label{Acccc}
\int_{A_z} G_\lambda(x,y')d\mu(y')\overset{\eqref{harnack4}}\leq e^{D_{1,1}}\int_{A_z} G_\lambda(x,z)d\mu(y')\leq C'e^{D_{1,1}} G_\lambda(x,z).
\end{equation}

By assumption \eqref{4.3} and Lemma \ref{AC!!}, we have the following inequality: For all $z$ in $O_1\cap\partial O_2$,
\begin{equation}\label{acc1}
\frac{1}{2}{C_{1/2}}<\int_{A_z}G_\lambda(z,y':O_1)d\mu(y').
\end{equation}
Denote $\sigma=\sigma_{O_1}$ and $\tau=\sigma_{O_1\backslash \overline{O_2}}$. We first replace $G_\lambda(x,y')$ by $G_\lambda(x,y':O_1)$ to obtain  the first inequality of \eqref{acc}. Then we apply Proposition \ref{MArkov}, where the second term in the right hand side of \eqref{exit} is zero since $A_z$ and $O_1 \backslash \overline{O_2}$ are disjoint. We now disregard the paths for which $z \neq \omega(\tau)$ to obtain the second inequality of \eqref{acc}.
\begin{equation}\label{acc}
\begin{split}
&\displaystyle\sum_{z\in O_1\cap\partial O_2}f(z)\int_{A_z}G_\lambda(x,y')d\mu(y')\geq\displaystyle\sum_{z\in O_1\cap\partial O_2}f(z)\int_{A_z}G_\lambda(x,y':O_1)d\mu(y')\\
&\overset{\substack{\text{Prop}\\\ref{MArkov}}}= \displaystyle\sum_{z\in O_1\cap\partial O_2}f(z) \displaystyle\mathbb{E}_x\biggl[1_{\tau< \sigma}\int_0^{\sigma}e^{\lambda t}{1}_{A_{\omega(\tau)}}(\omega(t))dt\biggr]\\
&\geq  \displaystyle\sum_{z\in O_1\cap\partial O_2}f(z) \displaystyle\mathbb{E}_x\biggl[1_{\tau< \sigma}1_{\omega(\tau)=z}\int_0^{\sigma}e^{\lambda t}{1}_{A_{\omega(\tau)}}(\omega(t))dt\biggr]\\
&\overset{\mathrm{def}}=\displaystyle\mathbb{E}_x\biggl[1_{\tau< \sigma}f(\omega(\tau))\int_0^{\sigma}e^{\lambda t}{1}_{A_{\omega(\tau)}}(\omega(t))dt\biggr]\\
&\overset{\eqref{4.1}}=\displaystyle\mathbb{E}_x\biggl[1_{\tau< \sigma}e^{\lambda\tau}f(\omega({\tau}))\mathbb{E}_{\omega(\tau)}\biggl[\int_0^{\sigma-\tau}e^{\lambda t}{1}_{A_{\omega(\tau)}}(\omega(t+\tau))dt\biggr]\biggr] \\
&=\displaystyle\mathbb{E}_x\biggl[1_{\tau< \sigma}e^{\lambda\tau}f(\omega(\tau))\int_{{A_{\omega(\tau)}}}G_\lambda(\omega(\tau),y':O_1)d\mu(y')\biggr]\\
&\overset{\eqref{acc1}}\geq\frac{1}{2}C_{1/2}\mathbb{E}_x\biggl[1_{\tau< \sigma}e^{\lambda\tau}f(\omega(\tau))\biggr]=\frac{1}{2}C_{1/2}\int_{O_1\cap\partial O_2}f(z)d\eta^{\lambda,O_1\cap\partial O_2}(z).
\end{split}
\end{equation}
The third equality of \eqref{acc} uses strong Markov property \eqref{4.1}. The inequality \eqref{acc1} shows the third inequality of \eqref{acc}.
Since $$\sum_{z\in O_1\cap\partial O_2}f(z)\int_{A_z} G_\lambda(x,y')d\mu(y')\leq C'e^{D_{1,1}}\sum_{z\in O_1\cap\partial O_2} G_\lambda(x,z)f(z)$$ by \eqref{Acccc}, we have the inequality \eqref{MARkov}.\end{proof}
Using strong Markov property, we obtain the first inequality of \eqref{1.1.0}.
\begin{thm}\label{lAncona} There exists a constant $C$ such that for all $\lambda\in(0,\lambda_0]$ and $x,z$ in $B(y,1)^c$,
\begin{equation}\label{AC5}
C^{-1}G_\lambda(x,y)G_\lambda(y,z)\leq G_\lambda(x,z).
\end{equation} 
\end{thm}
\begin{proof} Suppose that $y\in V$. Denote $B_y=B(y,1/2)\backslash\overline{B(y,1/4)}$ and $$c_y=\inf_{y_1\in B_y}\inf_{t\in[0,1]} E_{y_1}[1_{t+1\leq\tau\leq t+2}e^{\lambda \tau}],$$
where $\tau=\sigma_{\overline{B(y,1/8)}^c}$. 
 By strong Markov property \eqref{4.1}, for $k\geq 0$, 
\begin{equation}\label{4.9}
\begin{split}
&\int_{k}^{k+1}\mathbb{E}_x[1_{k+2\leq\tau\leq k+3}e^{\lambda\tau}]ds\geq\int_{k}^{k+1}\mathbb{E}_x\left[1_{k+2\leq\tau\leq k+3}e^{\lambda\tau}1_{B_y}(\omega(s))\right]ds\\ &\overset{\eqref{4.1}}= \int_{k}^{k+1}\mathbb{E}_x\left[e^{\lambda s}1_{B_y}(\omega(s))\mathbb{E}_{\omega(s)}\left[1_{k+2-s\leq \tau-s\leq k+3-s}e^{\lambda(\tau-s)}\right]\right]ds\\ &\geq c_y\int_{k}^{k+1}\mathbb{E}_x\left[e^{\lambda s}1_{B_y}(\omega(s))\right]ds
\end{split}
\end{equation}
 By Proposition \ref{MArkov} applied to $O_1=\X, O_2=B(y,1/8), y=z,$ we obtain the first inequality of \eqref{4.100}. Harnack inequality \eqref{harnack4} shows the second inequality of \eqref{4.100}. We disregard the paths $\omega$ satisfying $\tau(\omega)\leq 3$ to obtain the third inequality of \eqref{4.100}. Using \eqref{4.9}, we have the fourth inequality of \eqref{4.100}.
\begin{equation}
\begin{split}\label{4.100}&{G_\lambda(x,z)}{} \overset{\substack{\text{Prop}\\ \ref{MArkov}}}\geq \mathbb{E}_x\left[1_{\tau<\infty}e^{\lambda\tau} {G_\lambda(\omega({\tau}),z)}{}\right]\overset{\eqref{harnack4}}\geq e^{-D_{1,1}}\mathbb{E}_x[1_{\tau<\infty}e^{\lambda\tau}]G_\lambda(y,z)\\
 &\geq e^{-D_{1,1}}\sum_{k=0}^\infty\mathbb{E}_x[1_{k+2\leq\tau\leq k+3}e^{\lambda\tau}]G_\lambda(y,z)\\
 &\overset{\eqref{4.9}}\geq c_ye^{-D_{1,1}}\sum_{k=0}^\infty\mathbb{E}_x\left[\int_k^{k+1}e^{\lambda s}1_{B_y}(\omega(s))ds\right]G_\lambda(y,z)\\
 &=c_ye^{-D_{1,1}}\mathbb{E}_x\left[\int_0^\infty e^{\lambda s}1_{B_y}(\omega(s))ds\right]G_\lambda(y,z)\\
 &\overset{\text{Def}}=c_ye^{-D_{1,1}}\int_{B_y}G_\lambda(x,y')d\mu(y')G_\lambda(y,z).
 \end{split}
\end{equation} 
By compactness of $X$, $c_V=\underset{y\in V} \inf c_y$ is finite and positive. The measure of $B_y$ is at least $1/4$. Applying  Harnack inequality {\eqref{harnack4}} again, we have
 \begin{equation}\label{4.9.0}G_\lambda(x,z)\overset{\eqref{harnack4}}\geq \frac{c_V}{4}e^{-2D_{1,1}}G_\lambda(x,y)G_\lambda(y,z).\end{equation}
Suppose $y\notin V$. We first obtain the inequality \eqref{4.9.0} for the closest vertex to $y$. Using Harnack inequality \eqref{harnack4}, we have the inequality \eqref{4.9.0} for $y$.
 \end{proof}

By Proposition \ref{hk} and the integration by substitution, the following proposition holds as in \cite{LL}.
\begin{prop}\label{G11}  For $\lambda \in [0, \lambda_0)$, for any two distinct points $x, y$ in $ \widetilde{X}$,
\begin{equation}\label{G1}
\frac{\partial}{\partial \lambda} G_\lambda(x,y) =\displaystyle \int_{ \widetilde{X}} G_\lambda(x,z) G_\lambda (z,y) dz.
\end{equation}
\end{prop}

The inverse operator of $(-\Delta-\lambda I)$ is described by the Green function $G_\lambda(x,y)$ (see Corollary \ref{g}). Hence, for two distinct points $x,y$ in $\X$, the derivative $\frac{\partial}{\partial \lambda}G_\lambda(x,y)$ converges at $\lambda\in (0,\lambda_0)$.
\begin{defn}For points $x,y,z$ in a metric space $(Y,d_Y)$, the \textit{Gromov product} $(y|z)_x$ of $y$ and $z$ at $x$ is defined by $$(y|z)_x=\frac{1}{2}\{d_Y(x,y)+d_Y(x,z)-d_Y(y,z)\}.$$ 
 \end{defn}
 
For $y\in \X$, fix an element $\gamma_y$ of $\G$ such that $ y \in \gamma_y \overline{F}$.
\begin{lem}\label{Alpha}\emph{(\cite{G}, Lemma 2.4)} 
There exists a positive constant $D$ such that for all $x\in \overline{F}$ and $y,z\in \X$, there exists an element $\gamma(y,z) \in \Gamma$ satisfying $d_{\G}(e,\gamma(y,z))\leq D$ and \begin{equation}\label{tree1}d(x,y)+d(x,z)-3diam(F) \leq d(x,\gamma_y\gamma(y,z)z).\end{equation}
\end{lem}
\begin{proof} The main idea of the proof is from Lemma 2.4 of \cite{G}. We present the proof for completeness. By Theorem 2.12 in \cite{GH}, there exists a constant $C$ satisfying the following property. For all points $x_1,x_2,x_3,x_4\in \X$, there exists a map $\Phi$ from $\{x_1,x_2,x_3,x_4\}$ to some metric tree $T$ satisfying
\begin{equation}
d(x_i,x_j)-C\leq d_T(\Phi(x_i),\Phi(x_j))\leq d(x_i,x_j)
\end{equation}
for all $x_i,x_j\in  \{x_1,x_2,x_3,x_4\}$.
\begin{figure}[htbp]
\begin{center}
\includegraphics[width=90mm]{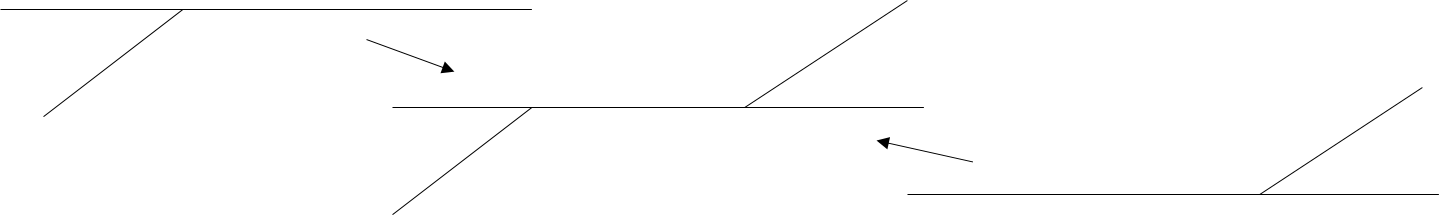}
\put(-280,35){$\Phi( x)$} \put(-270,10){$\Phi(\gamma_y^{-1}x)$} \put(-160,35){$\Phi(\alpha^{N'}x)$}\put(-180,30){$\gamma_y$}
\put(-220,18){$\Phi(\gamma_y x)$} \put(-210,-5){$\Phi( x)$}  \put(-90,18){$\Phi(\gamma_y \alpha^{N'} x)$} \put(-95,37){$\Phi(\gamma_y \alpha^{N'} \gamma_z x)$} \put(-140,-3){$\Phi(\alpha^{-N'} x)$} \put(-0,-5){$\Phi(x)$}\put(-0,25){$\Phi(\gamma_z x)$}\put(-130,10){$\gamma_y\alpha^{N'}$}
\end{center}
\caption{Tree approximation}\label{treea}
\end{figure}
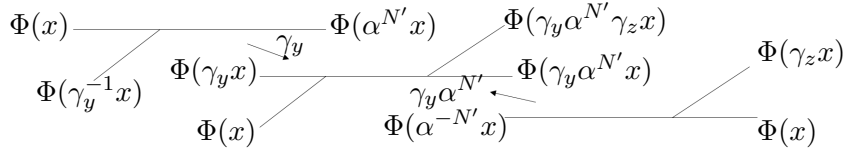

Since $\G$ is non-elementary, there exist two elements $\alpha$ and $\beta$ in $\G$ such that the sequences $\{\alpha^{\pm n} x\}$ and $\{\beta^{\pm n} x\}$ converge to 4 distinct points $\alpha_{\pm}$ and $\beta_{\pm}$ in $\partial \X$, respectively. Note that $\alpha_{\pm}$ and $\beta_{\pm}$ are independent of the choice of $x\in\overline{F}$. Choose small disjoint neighborhoods $V(\alpha_{\pm})$ and $V(\beta_{\pm})$ of $\alpha_{\pm}$ and $\beta_{\pm}$ in $\X \cup \partial \X$, respectively. Choose $N$ satisfying $\alpha^{\pm n} x\in V(\alpha_{\pm})$ and $\beta^{\pm n} x\in V(\beta_{\pm})$ for all $n\geq N$. For any $\gamma x\in V(\alpha_+)^c$ and $n\geq N$, there exists a constant $K_{\alpha_+}$ such that  the $(\Phi(\gamma x)|\Phi(\alpha^n x))_{\Phi(x)}\leq K_{\alpha_+}$. Otherwise, we obtain sequences $\{\gamma_k x\}\subset V(\alpha_{+})^c$ and $\{a^{n_k} x\}\subset V(\alpha_+)$ with $\lim_{k\rightarrow\infty} (\gamma_k x|a^{n_k} x)_x=\infty$. The sequence $\{\gamma_k x\}$ converges to $\alpha_+$, which is a contradiction. Denote by $K$ the number satisfying this property for $\alpha_\pm$ and $\beta_\pm$. Fix $N'\geq N$ with $d(a^{N'} x,x)\geq 4K+3C$ for all $a\in \{\alpha^{\pm1},\beta^{\pm1}\}$. 
Without loss of generality, we may assume that $\gamma_y^{-1}x\in V(\alpha_+)^c$ and $\gamma_z x\in V(\alpha_-)^c$. We have
\begin{equation}
\begin{split}
\nonumber d&(x,\gamma_y \alpha^{N'} z)\geq d(x,\gamma_y \alpha^{N'} \gamma_z x)-diam(F)\geq d_T(\Phi(x),\Phi(\gamma_y \alpha^{N'} \gamma_z x))-diam(F)\\
\nonumber=&d_T(\Phi(x),\Phi(\gamma_y x))+d_T(\Phi(\gamma_y x),\Phi(\gamma_y \alpha^{N'} x))+d_T(\Phi(\gamma_y \alpha^{N'} x),\Phi(\gamma_y \alpha^{N'} \gamma_z x))\\
\nonumber&-2(\Phi(x)|\Phi(\gamma_y \alpha^{N'} x))_{\Phi(\gamma_y x)}-2(\Phi(\gamma_y x)|\Phi(\gamma_y \alpha^{N'} \gamma_zx))_{\Phi(\gamma_y \alpha^{N'}x)}-diam(F)\\
\nonumber\geq& d(x,\gamma_y  x)+d(\gamma_y x,\gamma_y \alpha^{N'} x)+d(\gamma_y \alpha^{N'} x,\gamma_y \alpha^{N'} \gamma_z x)-3C-4K-diam(F)\\
\nonumber \geq& d(x,\gamma_y  x)+d(x,\gamma_z x)-diam(F)\geq d(x,y)+d(x,z)-3diam(F).
\end{split}
\end{equation}
In the other cases, it is possible to find an elements $\gamma(y,z)\in \{\alpha^{\pm N'},\beta^{\pm N'}\}$ satisfying \eqref{tree1}. Put $D=\max\{d_\Gamma(e,\gamma):\gamma\in \{\alpha^{\pm N'},\beta^{\pm N'}\}\}$. Then we complete the proof.
\end{proof}
 Using Lemma \ref{Alpha}, we prove the following which is the analog of Lemma 2.5 in \cite{G}. 

\begin{prop}\label{volume} 
Denote $V_n(v):=\{w\in V: n<d(v,w)\leq n+1\}.$ There exists a constant $C$ such that for any $v\in V$ and integer $n\geq 3diam(F)+1$, 
$$\sum_{w\in V_n(v)}G_{\lambda_0}^2(v,w) \leq C.$$
\end{prop}
\begin{proof}By Harnack inequality \eqref{harnack4}, for any $w\in e$ with $d(v,w)\geq diam(F)$, there exists a constant $C$ such that
\begin{equation}\label{eq:4.11}
G_\lambda^2(v,w)\leq \frac{C}{l_m}\int_0^{l_e} G_\lambda^2(v,e_s)ds,
\end{equation}
where $d(i(e),e_s)=s.$ 
By Harnack inequality \eqref{harnack4}, Proposition \ref{G11} and the inequality \eqref{eq:4.11}, for all $\lambda \in [0,\lambda_0)$ and for all $y \in B(v,1)$ with $v\neq y$,
\begin{equation}\label{4.2.17}
\begin{split}
\displaystyle  \sum_{n>diam(F)}\sum_{w\in V_n(v)} G_{\lambda}^2(v,w) &\overset{\eqref{eq:4.11}}\leq \sum_{n>diam(F)}\sum_{w\in V_n(v)}\sum_{\substack{e\in E\\w\in e}} \frac{C}{l_m}\int_0^{l_e} G_\lambda^2(v,e_s)ds\\
&\leq C\frac{2}{l_m}\displaystyle\int_{\widetilde{X}-B(v,1)} G_\lambda ^2 (v,z)d\mu(z)\\
&\overset{\eqref{harnack4}}\leq C  \displaystyle \frac{2}{l_m}\int_{\widetilde{X}-B(v,1)} G_\lambda (v,z) G_\lambda (z,y) d\mu(z)\overset{\substack{\text{Prop}\\ \ref{G11}}}<\infty,
\end{split}
\end{equation}
thus the sum $\sum_{w\in V_n(v)}G^2_{\lambda}(v,w)$ is bounded above for all $\lambda\in[0,\lambda_0)$.
 By Lemma \ref{AC!!} and Lemma \ref{Alpha}, there exists a constant $C$ such that for any $w_1 \in V_m(v)$ and $w_2 \in V_n(v)$, 
 \begin{equation}\label{a!@}
 G_\lambda(v,\gamma(w_1,w_2)v)\geq C.
 \end{equation}
For any $w_1 \in V_m(v)$ and $w_2 \in V_n(v)$, by the construction of $\gamma(w_1,w_2)$, the point $\gamma_{w_1} v$ and $\gamma_{w_1}\gamma(w_1,w_2) v$ are contained in the $K+2C$-neighborhood of the geodesic from $v$ to $\gamma_{w_1}\gamma(w_1,w_2)w_2$ (see Figure \ref{treea}), where $K$ is from the proof of Lemma \ref{Alpha} and $C$ is the upper bound of the radii of the inscribed circles of geodesic triangles in $\X$. Using Harnack inequality \eqref{harnack4}, the inequality \eqref{a!@} and Theorem \ref{lAncona} in order, we obtain 
 \begin{equation}\label{multi}
 \begin{split}
&G^2_{\lambda}(v,w_1)G^2_\lambda(v,w_2) =G^2_{\lambda}(v,w_1)G^2_\lambda(\gamma_{w_1}\gamma(w_1,w_2) v,\gamma_{w_1}\gamma(w_1,w_2) w_2)\\
&\leq C G^2_\lambda (v,\gamma_{w_1} v) G^2_\lambda(\gamma_{w_1} v,\gamma_{w_1}\gamma(w_1,w_2) v)G^2_\lambda(\gamma_{w_1}\gamma(w_1,w_2) v,
\gamma_{w_1}\gamma(w_1,w_2) w_2)\\
&\leq C G^2_\lambda(v,\gamma_{w_1}\gamma(w_1,w_2)w_2).
\end{split}
\end{equation}
 Since $d_{\G}(e,\gamma(w_1,w_2))\leq D$ for all $w_1 \in V_m(v)$ and $w_2 \in V_n(v)$,
$$T_1:=m+n-3diam(F)\leq d(v,\gamma_{w_1}\gamma(w_1,w_2)w_2)\leq T_2:=m+n+2+D\,diam(F).$$ 
For any vertex $w'$ with $T_1\leq d(v,w')\leq T_2$, the number of two pairs $(w_1,w_2)$ in $V_m(v)\times V_n(v)$ satisfying $w'=\gamma_{w_1}\gamma(w_1,w_2)w_2$ is uniformly bounded. It follows from the fact that $\gamma_{w_1} v$ and $\gamma_{w_1}\gamma(w_1,w_2) v$ are in the $K+2C$ neighborhood of the geodesic from $v$ to $w'$, $$m-diam(F) \leq d(v,\gamma_{w_1} v)\leq m+diam(F)+1$$ and $n-1\leq d(\gamma_{w_1}\gamma(w_1,w_2) v, w')\leq n.$

By the inequality \eqref{multi}, 
\begin{equation}\label{sphere}
\displaystyle\sum_{w_1 \in V_m(v)}\sum_{w_2\in V_n(v)}G^2_{\lambda}(v,w_1)G^2_\lambda(v,w_2)
\leq C\displaystyle \sum_{i=\llcorner T_1\lrcorner}^{\ulcorner T_2\urcorner} \sum_{w'\in V_i(v)}G^2_{\lambda}(v,w').
\end{equation}
Let $M_\lambda$ be the supremum of $\sum_{w\in V_n(v)}G^2_{\lambda}(v,w)$. By the inequality \eqref{sphere}, the following holds: 
$$M_{\lambda}^2 \leq C \sum_{i=\llcorner T_1\lrcorner}^{\ulcorner T_2\urcorner} M_{\lambda}\leq C(T_2-T_1+2) M_{\lambda}.$$
Hence, $M_\lambda\leq C(T_2-T_1+2)=C(D+3)diam(F)+4C$ for any $\lambda\in [0,\lambda_0)$. Suppose that $M_{\lambda_0}>C(T_2-T_1)$. There exists $n$ satisfying $\sum_{w\in V_n(v)} G_{\lambda_0}^2(v,w)>C(T_2-T_1).$ By the continuity of $G_\lambda$, there exists a constant $\lambda \in [0,\lambda_0)$ satisfying $$\sum_{w\in V_n(v)} G_{\lambda}^2(v,w)>C(T_2-T_1).$$ This is a contradiction. Thus $M_{\lambda_0}\leq C(T_2-T_1).$
\end{proof}

Using strong Markov property as in \cite{LL}, we obtain the following proposition.
\begin{prop}\emph{(\cite{LL}, Proposition 8.6)} \label{smv} Let $O_1$, $O_2$, and $O_3$ be open sets in $\X$ satisfying ${O_3}\subset {O_2}\subset O_1$. Then the following equation holds: for any $x\in O_1\backslash \bar{O_2}$,
\begin{displaymath}
\int_{\partial O_3} f(z)d\eta_x^{\partial O_3}(z)=\int_{\partial O_2}\left(\int_{\partial O_3} f(z)d\eta_y^{\partial O_3}(z)\right)d\eta_x^{\partial O_2}(y).
\end{displaymath}
\end{prop}

 The following lemma is analogous to Lemma 2.6 in \cite{G}, which is the main technical part of the proof of Theorem \ref{rancona}.
\begin{lem}\label{acnn} Let $y$ be a point on a geodesic $[x,z]$ between $x$ and $z$. There exist constants $\varepsilon>0$ and $R_0 >0$ such that for all   $d(x,y) \geq r, \ d(y,z) \geq r$ and $r \geq R_0$,
\begin{equation}\label{prean}
G_{\lambda_0}(x,z:\overline{B(y,r)}^c) \leq 2^{e^{-\varepsilon r}}.
\end{equation}

\end{lem}

\begin{proof}
\underline{\textit{Step 1. Geometric argument using quasi isometry:}} In this step, using a quasi-isometry $\Psi$ from $\widetilde{X}$ to a $n$-dimensional hyperbolic space $\mathbb{H}^n$, we construct a sequence of subsets $A_i$ of $\widetilde{X}$ which is used to  decompose $G_{\lambda_0}(x,z:\overline{B(x,r)}^c)$. Since the relative $\lambda$-Green function $G_{\lambda_0}(x,z:\overline{B(y,r)}^c)=0$ if $\overline{B(y,r)}^c$ is disconnected, we may assume that $\overline{B(x,r)}^c$ is connected for all $r$. Since the $\G$-action is cocompact, by Theorem 10.2 in \cite{BS}, there exist a map $\Psi$ from $ \widetilde{X}$ to a convex subset $Y$ of a hyperbolic space $\mathbb{H}^n$ and positive constants $L$ and $k$ such that for all $x,y\in  \widetilde{X}$,
$$|Ld(x,y)-d_{\mathbb{H}^n}(\Psi(x),\Psi(y))|\leq k$$
and the $k$-neighborhood of $\Psi( \widetilde{X})$ is contained in $Y.$ Since the image $\Psi([x,z])$ of the geodesic $[x,z]$ is a quasi-geodesic in $\mathbb{H}^n$, it is contained in the $K$-neighborhood of the geodesic $g$ from $\Psi(x)$ to $\Psi(z)$ in $\mathbb{H}^n$ (see Figure \ref{AGI}). Choose a point $o$ on the geodesic $g$ in $\mathbb{H}^n$ such that $d_{\mathbb{H}^n}(\Psi(y),o)<K$.  

 Let $a,b$ be points in $\X$ and let $i_{a,b}$ be the radius of the inscribed circle of the geodesic triangle $\triangle o\Psi(a)\Psi(b)$ in $\mathbb{H}^n$, which is bounded by a universal constant, say $R_1$ (Proposition II.1.17 in \cite{BH}). Choose a positive constant $\kappa<L$. Denote $$R_2=\max\left\{R_1,\frac{3k+2K}{L}+3l_M,\frac{3l_ML+3k+K+\log(8\tanh R_1)}{2(L-\kappa)}\right\}.$$ By Theorem 7.11.2.(i) in \cite{Be}, \begin{equation}\label{anglehy}\tanh i_{a,b}=\sinh  ((\Psi(a)|\Psi(b))_o) \tan \displaystyle\frac{1}{2}\angle_o\Psi(a)\Psi(b).\end{equation}
 By \eqref{anglehy}, for all $a,b$ in $\X$ with $(\Psi(a)|\Psi(b))_o>\log \sqrt{{4+8\tanh^2 R_2}}$,
 \begin{equation}\label{anglehy2}
 \tan \angle_O\Psi(a)\Psi(b)=\frac{2\sinh ((\Psi(a)|\Psi(b))_o) \tanh i_{a,b}}{\sinh^2 ((\Psi(a)|\Psi(b))_o)- \tanh^2 i_{a,b}}\leq 8(\tanh i_{a,b}) e^{-(\Psi(a)|\Psi(b))_o}.
 \end{equation}
  By the triangle inequality, for all edge $e$ with $d(i(e),y)>R_2$ and points $a,b \in e$, 
\begin{equation}\label{anglehy3}
\begin{split}
  2(\Psi(a)|\Psi(b))_o&=d_{\mathbb{H}^n}(\Psi(a),o)+d_{\mathbb{H}^n}(\Psi(b),o)-d_{\mathbb{H}^n}(\Psi(a),
\Psi(b))\\ 
&\geq{L}(d(a,y)+d(b,y)-d(a,b))-{3k}-2K\\ 
&\geq{L}(2d(i(e),y)-3l_M)-{3k}-2K>0.
\end{split}
\end{equation}
By \eqref{anglehy2} and \eqref{anglehy3}, for all edge $e$ with $d(i(e),y)>R_2$,
\begin{equation}\label{3.17}
\theta_e:=\max\{\angle_o\Psi(a)\Psi(b):a,b\in e\}\leq \tan \theta_e\leq e^{-\kappa d(i(e),y)}.
\end{equation}
\begin{figure}[htbp]
\begin{center}
{\includegraphics[width=110mm]{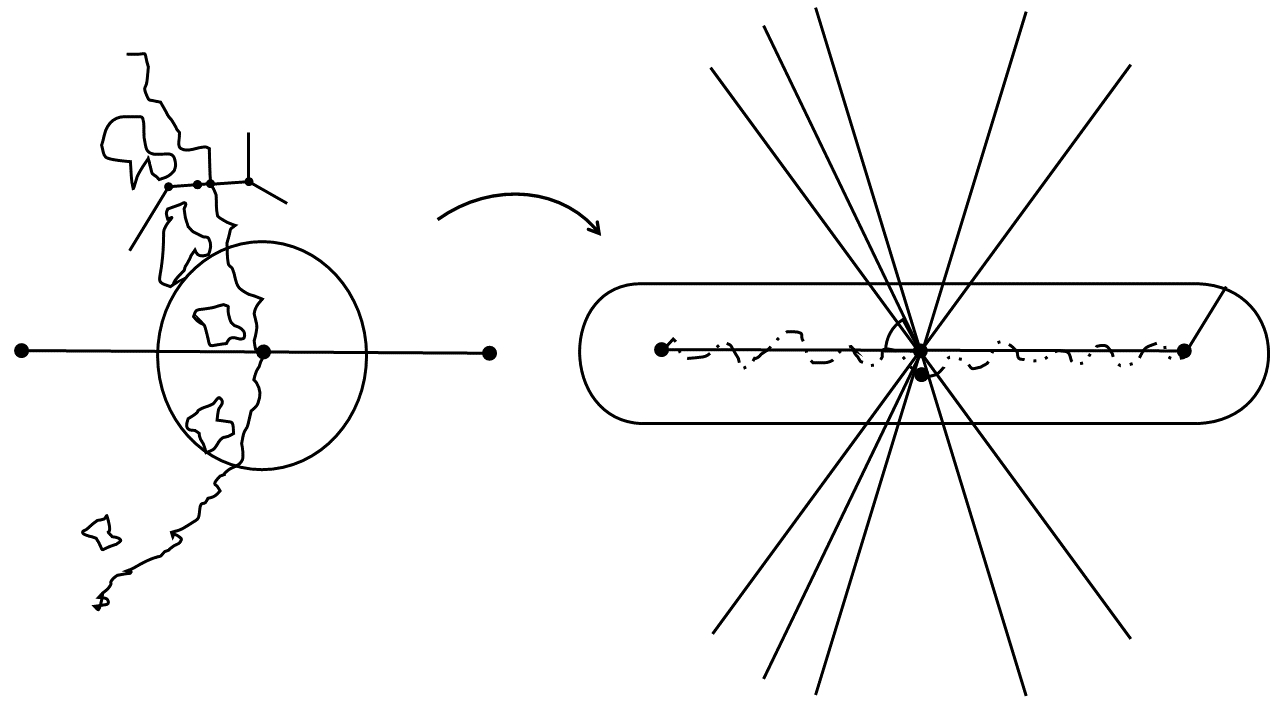}}
\put(-317,87){{$x$}} \put(-160,80){\small{$\Psi(x)$}}\put(-90,97){\large{$o$}}
\put(-245,83){{$y$}}\put(-100,75){\small{$\Psi(y)$}}\put(-261,134){\small{$u_i$}}
\put(-190,85){{$z$}}\put(-40,80){\small{$\Psi(z)$}}\put(-266,125){\small{$w$}}
\put(-260,180){\large{$\widetilde{X}$}}\put(-95,180){\large{$\mathbb{H}^n$}}
\put(-300,165){{$\partial A_i(\theta_i)$}}\put(-182,158){{$\theta=\frac{2i-1}{N}$}}\put(-178,173){{$\theta=\frac{2i-1+\theta_i}{N}$}}\put(-130,180){\large{$\theta=\frac{2i}{N}$}}
\put(-70,175){\large{$\theta=\frac{2i+1}{N}$}}\put(-50,160){\large{$\theta=\frac{2i+2}{N}$}}
\put(-190,135){\large{$\Psi$}}\put(-15,90){$K$}
\end{center}
\caption{Ancona-Gou\"ezel inequality}\label{AGI}
\end{figure}

Fix $\epsilon<\kappa$ and $r>R_2$ satisfying $e^{-\kappa r}<e^{-\epsilon r}/4$. Denote $N=\llcorner e^{\epsilon r}\lrcorner$. Denote $A_0=\{x\}$, $A_{N+1}=\{z\}$ and let $A_i(\theta)$ be the connected component containing $z$ of the set $$\{u\in  \overline{B(y,r)}^c:\angle_o\Psi(x)\Psi(u)> (2i-1)/N+\theta\}.$$ Fix an N-tuple $(\theta_1,\cdots,\theta_N)\in[0,1/N]^N$. Denote $A_i=A_i(\theta_i)$.  \\
  \underline{\textit{Step 2. Decomposition of $G_{\lambda_0}(x,z:\overline{B(x,r)}^c)$:}} In this step, using Proposition \ref{smv} and $A_i$, we represent  the $\lambda$-Green function $G_{\lambda_0}(x,z:\overline{B(x,r)}^c)$ as an integral. For any point $u \in  \overline{A_i}^c$, every continuous path in $ \overline{B(x,r)}^c$ from $u$ to $z$ must go through $\partial A_{i}$. Thus the second term of the right hand side of \eqref{exit} is zero. For the convenience, denote $\partial A_{i,r}:=\partial A_{i}\cap \overline{B(x,r)}^c$. By the definition of $\eta_{u_i}^{\lambda_0,\partial A_{i+1,r}}$, for any $u_i\in \partial A_{i,r},$ we have
  \begin{equation}\label{3.13.1}
 G_{\lambda_0}(u_i,z:\overline{B(x,r)}^c)=\int_{\partial A_{i+1,r}}G_{\lambda_0}(u_{i+1},z:\overline{B(x,r)}^c)d\eta_{u_i}^{\lambda_0,\partial A_{i+1,r}}(u_{i+1}).
  \end{equation}
 By Propositon \ref{smv} and \eqref{3.13.1},
 \begin{equation}\label{3.18}
 \begin{split}
  &G_{\lambda_0}(x,z:\overline{B(x,r)}^c)\\
  &=\int_{\partial A_{N,r}}G_{\lambda_0}(u_{N},z:\overline{B(x,r)}^c)d\eta_{x}^{\lambda_0,\partial A_{N,r}}(u_{N})\\
  &=\int_{\partial A_{N-1,r}}\int_{\partial A_{N,r}}G_{\lambda_0}(u_{N},z:\overline{B(x,r)}^c)d\eta_{u_{N-1}}^{\lambda_0,\partial A_{N,r}}(u_N)\eta_{x}^{\lambda_0,\partial A_{N-1,r}}(u_{N-1})\\
  &=\int_{\partial A_{1,r}}\cdots\int_{\partial A_{N,r}}G_{\lambda_0}(u_{N},z:\overline{B(x,r)}^c)d\eta_{u_{N-1}}^{\lambda_0,\partial A_{N,r}}(u_N)\cdots\eta_{x}^{\lambda_0,\partial A_{1,r}}(u_{1})\\
  &\leq\int_{\partial A_{1,r}}\cdots\int_{\partial A_{N,r}}G_{\lambda_0}(u_{N},z)d\eta_{u_{N-1}}^{\lambda_0,\partial A_{N,r}}(u_N)\cdots\eta_{x}^{\lambda_0,\partial A_{1,r}}(u_{1}).
  \end{split}
 \end{equation}
 Since $A_i$ is connected, $\mu((B({u_i},1)\cap A_i)\backslash B({u_i},1/2))\geq \frac{1}{2}.$ By Proposition \ref{acn}, for any $u_i\in \partial A_{i,r}$, there exists a constant $C_1$ such that
 \begin{equation}\label{3.19}
 \int_{\partial A_{i+1,r}}G_{\lambda_0}(u_{i+1},z)d\eta_{u_{i}}^{\lambda_0,\partial A_{i+1,r}}(u_{i+1})\leq C_1\sum_{u_{i+1}\in\partial A_{i+1,r}}G_{\lambda_0}(u_i,u_{i+1})G_{\lambda_0}(u_{i+1},z).
 \end{equation}
 As in \cite{G} and \cite{LL}, the operator $L_{i}:{l}^2(\partial A_{i+1,r})\rightarrow l^2(\partial A_{i,r})$ is defined by 
$$L_{i}f(u_i)=\sum_{u_{i+1}\in\partial A_{i+1,r}}G_{\lambda_0}(u_i,u_{i+1})f(u_{i+1}).$$
Let $||L_i||_{op}$ be the operator norm of $L_i$. Applying \eqref{3.18} and \eqref{3.19}, 
\begin{align*}\label{prean2}
&G_{\lambda_0}(x,z:\overline{B(x,r)}^c)\\
&\leq C_1^N\displaystyle\sum_{u_1\in \partial A_{1,r}}\cdots\sum_{u_N\in \partial A_{N,r}}
G_{\lambda_0}(x,u_1)G_{\lambda_0}(u_1,u_2)\cdots G_{\lambda_0}(u_N,z)\nonumber\\
&=C_1^N(G_{\lambda_0}(x\cdot),L_1L_2\cdots L_NG_{\lambda_0}(\cdot,z))_{l^2(\partial A_{1,r})}\nonumber\\
&\leq C_1^N||G_{\lambda_0}(x,\cdot)||_{l^2(\partial A_{1,r})}||L_1\cdots L_NG_{\lambda_0}(\cdot,z)||_{l^2(\partial A_{1,r})}\nonumber\\ 
&\leq C_1^N||G_{\lambda_0}(x,\cdot)||_{l^2(\partial A_{1,r})}||L_1||_{op}\cdots||L_N||_{op} ||G_{\lambda_0}(\cdot,z)||_{l^2(\partial A_{N,r})}.
\end{align*}
By Cauchy inequality,
$$||L_if||_{l^2(\partial A_{i,r})}^2\leq \sum_{\substack{u_i\in \partial A_{i,r}\\ u_{i+1}\partial A_{i+1,r}}}G^2_{\lambda_0}(u_i,u_{i+1})||f||_{l^2(\partial A_{i+1,r})}^2.$$
Denote $$f_i(\theta_1,\cdots, \theta_N):=\displaystyle\sum_{\substack {u_i \in \partial A_{i,r}(\theta_i) \\ u_{i+1}\in \partial A_{i+1,r}(\theta_{i+1})}}G_{\lambda_0}^2(u_i,u_{i+1}),$$
where $\partial A_{i,r}(\theta_i):=\partial A_{i}(\theta_i)\cap \overline{B(x,r)}^c$.
To complete the proof, it remains to find an $N$-tuple $(\theta_1,\cdots \theta_N) \in [0,1/N]^N$ such that for all $i$,
$$||L_i||_{op}\leq f_i(\theta_1,\cdots,\theta_N)^{1/2}<\frac{1}{2C_1}.$$
\underline{\textit{Step 3. Contribution of edges in $f_i(\theta_1,\cdots,\theta_N)$:}}  Given edges $e$ and $e'$, the function $G|_{e,e'}(\theta_{i},\theta_{i+1})$ is defined by
\begin{displaymath}
G|_{e,e'}(\theta_{i},\theta_{i+1})=\begin{cases} G_{\lambda_0}(x,y) &\text{if } x\in e\cap\partial A_{i}(\theta_{i}) \text{ and } y\in e'\cap\partial A_{i+1}(\theta_{i+1})\\
0 &\text{otherwise}.
\end{cases}\end{displaymath}
By applying Corollary \ref{harnack5} twice to edges $e$ and $e'$, 
\begin{equation}\label{3.23}
G|_{e,e'}(\theta_{i},\theta_{i+1})\leq e^{2D_{l_M,l_M}}G_{\lambda_0}(i(e),i(e')).
\end{equation}
Let $d\theta_i$ be Lebesgue measure on the interval $ [0,1/2N]$. By \eqref{3.17} and \eqref{3.23}, 
\begin{equation}\label{3.24}
\begin{split}
&\displaystyle\int_0^{1/2N}\int_0^{1/2N}G|_{e,e'}^2(\theta_{i},\theta_{i+1})d\theta_{i}d\theta_{i+1}\\  
&\overset{\eqref{3.23}}\leq e^{4D_{l_M,l_M}}G_{\lambda_0}^2(i(e),i(e'))\theta_{e}\theta_{e'}\\
&\overset{\eqref{3.17}}\leq e^{4D_{l_M,l_M}}G_{\lambda_0}^2(i({e}),i({e'}))e^{-\kappa(d({i(e)},y)+d({i(e'}),y))}.\end{split}
\end{equation}\\
\underline{\textit{Step 4. Counting $\gamma \in \G_i(v,w)$:}} Fix a fundamental domain $F$ containing $y$. Let $E_i$ be the set of edges in $ \widetilde{X}$ that intersect $\partial A_i(\theta)$ for some $\theta\in[0,1/N]$. By \eqref{3.24},
 \begin{equation}\label{3.25}
 \begin{split}
 &\displaystyle \int_0^{\frac{1}{N}}\int_0^{\frac{1}{N}}f_id\theta_i d\theta_{i+1}\\ 
 &=\displaystyle\sum_{\substack{e_{i} \in E_{i}\\e_{i+1}\in E_{i+1}}}\int_0^{1/2N}\int_0^{1/2N}G|_{e_{i},e_{i+1}}^2(\theta_{i},\theta_{i+1})d\theta_{i}d\theta_{i+1}\\
 &\leq C_2\displaystyle \displaystyle\sum_{\substack{e_{i} \in E_{i}\\e_{i+1}\in E_{i+1}}}e^{-\kappa d(i(e_i),i(e_{i+1}))}G_{\lambda_0}^2(i({e_i}),i(e_{i+1}))\\
 &= C_2\displaystyle \displaystyle\sum_{\substack{e_{i} \in E_{i}\\e_{i+1}\in E_{i+1}}}e^{-\kappa d(\gamma_{i(e_i)}^{-1}i(e_i),\gamma_{i(e_i)}^{-1}i(e_{i+1}))}G_{\lambda_0}^2(\gamma_{i(e_i)}^{-1}i({e_i})),\gamma_{i(e_i)}^{-1}i(e_{i+1}))
 \end{split}
 \end{equation}
for some constant $C_2$. It follows from inequality \eqref{3.17} that for any $e_i\in E_i$ and $w\in e_i$,
$$(8i-5)/4N\leq(2i-1)/N-\theta_{e_i}\leq\angle_o\Psi(x)\Psi(w)\leq 2i/N+\theta_{e_i}\leq(8i+1)/4N.$$
Denote $X_i=\{u\in  \overline{B(x,r)}^c:\angle_o\Psi(x)\Psi(u)\in[(8i-5)/4N,(8i+1)/4N]\}.$ If $e_i\in E_i$, then $e_i\subset X_i$. Let $[v,w]$ be a geodesic segment from a vertex $v$ to a vertex $w$ in $\widetilde{X}$. Denote  $$\Gamma_{i}(v,w):=\{\gamma\in \Gamma:\gamma v \in X_i \,\,\text{ and }\,\,\gamma w \in X_{i+1}\}.$$   
The vertex $\gamma_{i(e_i)}^{-1}i({e_i})$ in the fourth line of \eqref{3.25} is in $\overline{F}$. Thus we have
 \begin{equation}\label{3.26}\displaystyle \int_0^{\frac{1}{N}}\int_0^{\frac{1}{N}}f_id\theta_i d\theta_{i+1}\leq C_2\sum_{v\in V\cap \overline{F}}\sum_{w\in V} |\Gamma_{i}(v,w)|e^{-\kappa d(v,w)} G_{\lambda_0}^2(v,w).
  \end{equation} 
 The inequality \eqref{3.17} shows that for any $\gamma \in \G_i(v,w),$ $$e^{-\kappa(\gamma v|\gamma w)_y}\geq \angle_o \Psi(\gamma v)\Psi(\gamma w)\geq 1/4N.$$ and $\kappa(\gamma v|\gamma w)_y\leq 2\epsilon r $ for sufficiently large $r$. Since $d(y,\gamma v)\geq r$ and $d(y,\gamma w)\geq r$, for sufficiently small $\epsilon$, $d(\gamma v,\gamma w)\geq r$. By the hyperbolicity of $\X$, there exists a constant $C'$ such that $[\gamma v, \gamma w]$ intersects $B(y,C'\epsilon r)$ for all $\gamma\in \G_i(v,w)$. 
 
 Denote $h(r)=(r+2diam(F))/l_m.$ Let $id_\G$ be the identity of $\G$. Since for all $\gamma\in B(id_\G, h(C'\epsilon r))^c,$ $d(y,\gamma y)\geq C'\epsilon r$, $B(id_\G,h(C'\epsilon r))\overline{F}$ contains $B(y,C'\epsilon r)$. This implies that for any $\gamma y' \in B(y,h(C'\epsilon r))\cap [\gamma v,\gamma w]$, the element $\gamma$ is in $B(id_\G,h(2\epsilon r))\gamma_{y'}^{-1}$. Since the geodesic $[v,w]$  intersects at most $h(d(v,w))$ orbits of $\overline{F}$,  
 $|\Gamma_{i}(v,w)|\leq(1+h(d(v,w)))e^{C''\epsilon d(v,w)}$ for some $C''$.\\
 \underline{\textit{Step 5. Finding $\epsilon$ and $r$ using $\kappa$ in \eqref{3.17}:}} Choose $R_3>0$ so that for all $r>R_3,$ $e^{\kappa r/2}\geq (1+h(r+1))^2$. By \eqref{3.26} and Corollary \ref{volume}, we can choose $C_3$ such that for all $r>R_3$,
\begin{align*}
&\displaystyle \int_0^{\frac{1}{N}}\int_0^{\frac{1}{N}}f_id\theta_i d\theta_{i+1}\\
&\leq\displaystyle C_2\sum_{v\in V\cap \overline{F}}\sum_{n=\llcorner r\lrcorner}^\infty\sum_{w\in V_n(v)} \{1+h(n+1)\}^2e^{-\kappa n+2C''\epsilon(n+1)}G_{\lambda_0}^2(v,w)\\
&\leq \displaystyle C_3\sum_{n=\llcorner r\lrcorner}^\infty e^{-\kappa n/2+2C''\epsilon (n+1)}.
\end{align*}
Choose a sufficiently small $\epsilon$ with $\max\{2\epsilon(C''+1),2\epsilon\} \leq \kappa/4$. 
\begin{align*}
 &\displaystyle N^2\int_0^{\frac{1}{N}}\int_0^{\frac{1}{N}}f_id\theta_i d\theta_{i+1} \leq C_3\displaystyle\sum_{n=\llcorner r\lrcorner}^\infty e^{-\kappa n/2 +2\epsilon C''(n+1)+2\epsilon r}\\ 
 &\leq C_3\displaystyle\sum_{n=\llcorner r\lrcorner}^\infty e^{-\kappa n/2 +2\epsilon(C''+1)(n+1)}\\
&\leq C_3e^{\kappa/4}\displaystyle \sum_{n=\llcorner r \lrcorner}^\infty e^{-\kappa n/4}\leq \frac{C_3e^{\kappa/4-\kappa r/2}}{1- e^{-\kappa /4}}=C_4e^{-\kappa r/2}.
\end{align*}
Choose $R_4>0$ satisfying $e^{-2\epsilon r}/(8C_1^2C_4)>e^{-\kappa r/2}$ for all $r>R_4$,  
$$N^2\displaystyle\int\sum_if_id\theta_1\cdots\theta_N\leq1/(4C_1^2).$$ Put $R_0=\max\{R_2,R_3,R_4\}$. Then there exists an N-tuple $(\theta_1,\cdots,\theta_N)$ such that for any $i\in N$, $f_i(\theta_1,\cdots ,\theta_{N})\leq 1/(4C_1^2)$.\end{proof}
 Using Harnack inequality \eqref{harnack4}, Theorem \ref{lAncona}, and Lemma \ref{acnn}, we obtain Ancona inequality as in \cite{LL}.
\begin{thm}\label{rancona}Let $y$ be a point on a geodesic $[x,z]$ between $x$ and $z$. Suppose that $d(x,y)\geq 1$ and $d(y,z)\geq 1$. There exists a constant $C$ such that for all $\lambda\in(0,\lambda_0]$,
\begin{equation}\label{AC6}
G_\lambda(x,z)\leq CG_\lambda(x,y)G_\lambda(y,z).
\end{equation}
\end{thm} 
\subsection{Martin boundary}\label{sec:3.3} In this section, we show that the Gromov boundary coincides with its $\lambda$-Martin boundary for $\lambda\in [0,\lambda_0]$.
\begin{defn} The $\lambda$-\emph{Martin kernel} $K$ of $\X$ is defined as follows:
$$K_\lambda(x_0,x,y)=\frac{G_\lambda(x,y)}{G_\lambda(x_0,y)}.$$
The $\lambda$-\emph{Martin boundary} is the boundary of the image of the embedding defined by $y\mapsto K_\lambda(x_0,\cdot,y)$ on $\X$. 
\end{defn} 
 
Let $f$ and $g$ be functions on $ \widetilde{X}$. Denote by $f\asymp_c g$ if there exists a constant $c$ such that $f\leq cg$ and $g\leq cf$. 
 The following theorem is the analog of Theorem 4.6 in \cite{GL}. Unlike the proof in \cite{GL}, we prove the theorem without harmonic functions.
\begin{thm}\label{MB1} Let $[x,y]$ be a geodesic segment of length $n\geq 3$. Suppose that $d(x,y)\leq d(x',y')$ and $[x,y]$ is contained in the $r$-neighborhood of a geodesic segment $[x',y']$. Then there exist positive constant $C(r)$ and $\rho<1$ such that for all $\lambda$ in $[0,\lambda_0]$,
\begin{equation}\label{strongancona}
\displaystyle \left|\frac{G_\lambda(x,y)/G_\lambda(x',y)}{G_\lambda(x,y')/G_\lambda(x',y')}-1\right|\leq C\rho^n.
\end{equation}
The constant $C$ depends only on $r$.
\end{thm}
\begin{proof}
Let $x_1 \in [x,y]$ such that  
$$\min\{d(x,x_1),d(x',x_1)\}=2.$$ Let $x_k \in [x,y]$ with $d(x_1,x_k)=k-1$ for all integer $k< d(x_1,y)$ (see Figure \ref{SAI}).
\begin{figure}[h]\label{figure2}\begin{center}(100,-40)
\put(-130,-23){\includegraphics[width=70mm]{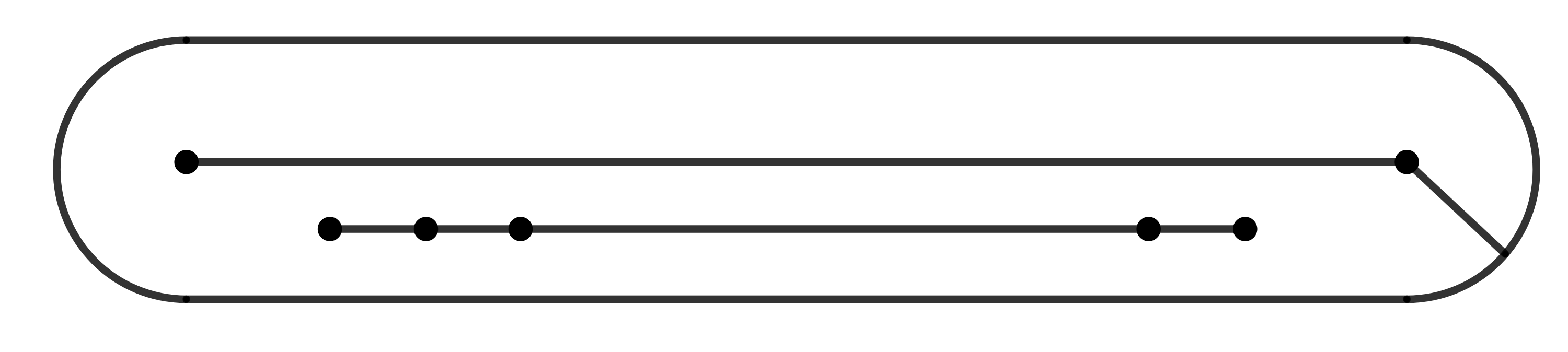}}
\put(-110,4){$x'$}\put(50,4){$y'$}\put(-95,-10){$x$}\put(-80,-14){$x_1$}\put(-68,-14){$x_2$}\put(0,-13){$x_{n-2}$}
\put(30,-10){$y$}\put(50,-12){$r$}
 \end{center}
\caption{Strong Ancona-Gou\"ezel inequality}\label{SAI}
\end{figure}  

Using \eqref{AC5}, \eqref{AC6}, and Harnack inequality \eqref{harnack4}, there exists a constant $C$ such that
\begin{equation} \label{AC7}\frac{G_\lambda(x,y)}{G_\lambda(x',y)}\asymp_{C} \frac{G_\lambda(x,x_k)}{G_\lambda(x',x_k)}\text{ and } \frac{G_\lambda(x,y')}{G_\lambda(x',y')}\asymp_{C} \frac{G_\lambda(x,x_k)}{G_\lambda(x',x_k)}.
\end{equation}

We first claim that for all integer $k< d(x_1,y)$,
\begin{equation}\label{HB}
 A(x,x',y,k):=\frac{G_\lambda(x,y)}{G_\lambda(x',y)}-\sum_{i=1}^k(1-\frac{1}{C})^{i-1} \frac{G_\lambda(x,x_i)}{G_\lambda(x',x_i)}\leq (1-\frac{1}{C})^k \frac{G_\lambda(x,y')}{G_\lambda(x',y')}
\end{equation}
and 
\begin{equation}\label{HB1}
 A(x,x',y,k)\geq -(1-\frac{1}{C})^k \frac{G_\lambda(x,y_k)}{G_\lambda(x',y_k)}.
\end{equation}
By \eqref{AC7}, \eqref{HB} and \eqref{HB1}, it follows that
\begin{equation}
\begin{split}\label{eq:59}
 \left|\frac{G_\lambda(x,y)}{G_\lambda(x',y)}-\frac{G_\lambda(x,y')}{G_\lambda(x',y')}\right|&\leq  \left|A(x,x',y,k)-A(x,x',y',k)\right|\\
&\leq |A(x,x',y,k)|+|A(x,x',y',k)| \leq 2{C}(1-\frac{1}{C})^k\frac{G_\lambda(x,y')}{G_\lambda(x',y')}.
\end{split}
\end{equation}
It remains to prove the claim. It is clear when $n=1$. Suppose that the inequality \eqref{HB} holds for all $k\leq n$. By induction, the following inequality holds:
\begin{align*}
A(x,x',y,k+1)&=A(x,x',y,k)-(1-\frac{1}{C})^k \frac{G_\lambda(x,x_{k+1})}{G_\lambda(x',x_{k+1})}\\
&\geq(1-\frac{1}{C})^k \frac{G_\lambda(x,y')}{G_\lambda(x',y')}-\frac{1}{C}(1-\frac{1}{C})^k \frac{G_\lambda(x,y')}{G_\lambda(x',y')}\\
&=(1-\frac{1}{C})^{k+1}\frac{G_\lambda(x,y')}{G_\lambda(x',y')}.
\end{align*}
Similarly, \eqref{HB1} holds for all $k$. By \eqref{eq:59}, we have \eqref{strongancona}.
\end{proof} 
 \begin{lem}\label{Zero} For any $\lambda\in [0,\lambda_0]$, $G_\lambda(x,y)$ goes to zero as $y$ goes to infinity.
 \end{lem}
 \begin{proof} The proof of lemma will use the idea in the proof of Theorem \ref{lAncona}.
  Suppose that there exist a constant $c>0$ and a sequence $\{y_n\}$ such that $$G_{\lambda_0}(x,y_n)\geq c \text{ and } \lim_{n\rightarrow \infty}d(x,y_n)=\infty.$$  By choosing a subsequence, we may assume that for all distinct two points $y_n$ and $y_{n'},$ $d(y_n,y_{n'})>8$ and $y_n$, $d(x,y_n)>8$.  Choose a point $y'$ with $d(x,y')=2$.  Since the action of $\G$ is cocompact, by Harnack inequality \eqref{harnack4}, we may assume that $y_n=\gamma_n y_1$ for some $\gamma_n\in \Gamma$. Denote $B_n=B(y_n,1),$ $B:=\bigcup\limits_{n=1}^{\infty} B_n,$  { and }$C_{n}= B(y_n,3)\backslash \overline{B(y_n,2)}.$
 Let $T$ be the first visit time of $\overline{B}$, i.e. $$T(\omega):=\inf\{t:\omega(t) \in \overline {B}\}.$$ 

\begin{figure}[h]
\begin{center}
{\includegraphics[width=80mm]{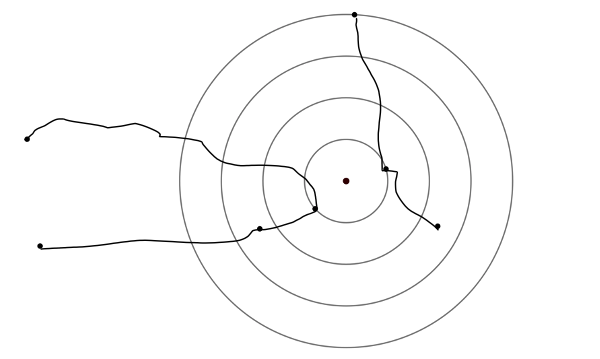}}
\put(-220,41){$y'$}\put(-197,92){{$\omega$}}\put(-146,55){{$\omega_2(s)$}}\put(-120,45){{$\omega(T)$}}\put(-224,82){{$x$}}\put(-100,58){{$y_n$}}\put(-66,40){{$z$}}\put(-90,134){{$\omega'(\tau_n)$}}\put(-82,74){{$\omega'(T)$}}
\end{center}
\caption{}\label{figure4}
\end{figure}
 Proposition \ref{MArkov} shows the first equation of \eqref{zeros4}. To obtain the second inequality of \eqref{zeros4}, we disregard paths $\omega$ satisfying $T(\omega)< 2.$
  \begin{equation}\label{zeros4}
\begin{split}
G_{\lambda_0}(x,y')
&\overset{\substack{\text{Prop}\\\ref{MArkov}}}=\mathbb{E}_x[{1}_{T<\infty}e^{\lambda_0 T}G_{\lambda_0}(\omega({T}),y')]+G_{\lambda_0}(x,y': \overline{B}^c)\\
&\geq\mathbb{E}_x[{1}_{T<\infty}e^{\lambda_0 T}G_{\lambda_0}(\omega({T}),y')]\\
&\geq\sum_{k=0}^\infty \mathbb{E}_x[{1}_{T\in [k+2,k+3)}e^{\lambda_0 T}G_{\lambda_0}(\omega({T}),y')]\\
&=\sum_{k=0}^\infty\sum_{n=1}^\infty \mathbb{E}_x[{1}_{T\in [k+2,k+3)}1_{\omega(T)\in \partial B_n}e^{\lambda_0 T}G_{\lambda_0}(\omega({T}),y')].\\
\end{split}
\end{equation}

We first show that there exists a positive constant $C$ such that for any $n$ and $k$,
\begin{equation}\label{zeros6}
\begin{split}
&\mathbb{E}_x[{1}_{T\in [k+2,k+3)}1_{\omega(T)\in \partial B_n}e^{\lambda_0 T}G_{\lambda_0}(\omega({T}),y')]\\
&\geq \int_k^{k+1} CG_{\lambda_0}(y_n,y')\mathbb{E}_x[1_{C_n}(\omega(s))e^{\lambda_0 s}]ds.
\end{split}
\end{equation}
Using \eqref{zeros6}, we will complete the proof.

Let $C_H=e^{-D_{2,2}},$ By construction of $D_{r,l}$, $e^{-D_{1,1}}\geq C_H$. The first inequality of \eqref{eq:61} follows from Harnack inequality \eqref{harnack4}. To obtain the second inequality of \eqref{eq:61}, we only consider paths satisfying that $\omega(s)\in C_n$ for some $s \in [k,k+1]$, $T(\omega)\in[k+2,k+3)$ and $\omega(T)\in \partial B_n$ (see Figure \ref{figure4}). The strong Markov property \eqref{4.1} shows the last equation of \eqref{eq:61}.
\begin{equation}\label{eq:61}
\begin{split}
&\mathbb{E}_x[{1}_{T\in [k+2,k+3)}1_{\omega(T)\in \partial B_n}e^{\lambda_0 T}G_{\lambda_0}(\omega({T}),y')]\\
&= \int_k^{k+1}\mathbb{E}_x[{1}_{T\in [k+2,k+3)}1_{\omega(T)\in \partial B_n}e^{\lambda_0 T}G_{\lambda_0}(\omega({T}),y')]ds\\
&\overset{\eqref{harnack4}}\geq \int_k^{k+1}C_HG_{\lambda_0}(y_n,y')\mathbb{E}_x[{1}_{T\in [k+2,k+3)}1_{\omega(T)\in \partial B_n}e^{\lambda_0 T}]ds\\
&\geq \int_k^{k+1}C_HG_{\lambda_0}(y_n,y')\mathbb{E}_x[1_{C_n}(\omega(s)){1}_{T\in [k+2,k+3)}1_{\omega(T)\in \partial B_n}e^{\lambda_0 T}]]ds\\
&\overset{\eqref{4.1}}=\int_k^{k+1}C_HG_{\lambda_0}(y_n,y')\\
&\quad\quad\quad\quad\quad\quad\times\mathbb{E}_x[1_{C_n}(\omega(s))e^{\lambda_0s}\mathbb{E}_{\omega(s)}[{1}_{T-s\in [k-s+2,k-s+3)}1_{\omega(T)\in \partial B_n}e^{\lambda_0 (T-s)}]]ds.\\
\end{split}
\end{equation}

Since for any $s\in [k,k+1],$ $0\leq k-s+1\leq 1,$ to find the constant $C$ satisfying \eqref{zeros6}, we show that there exists a positive constant $c_1$ such that  
\begin{equation}\label{zeros7}
\inf_{n}\inf_{z\in C_n}\inf_{t\in [0,1]}\mathbb{E}_z[{1}_{T\in [t+1,t+2)}1_{\omega'(T)\in \partial B_n}e^{\lambda_0 T}]]\geq c_1.
\end{equation}
 Let $\tau_n$ be the exit time of the ball ${B(y_n,4)}.$  For any $z\in C_n$ and $\omega'\in \Omega_z$ with $\tau_n(\omega')>T(\omega'),$ $\omega'(T)\in \partial B_n$ (See Figure \ref{figure4}). Thus for any $t\in [0,1]$ and $z\in C_n,$
\begin{equation}\label{eq:611}
E_{z}[1_{T\in[t+1, t+2)}1_{T<\tau_n}e^{\lambda_0 T}]\leq\mathbb{E}_z[{1}_{T\in [t+1,t+2)}1_{\omega'(T)\in \partial B_n}e^{\lambda_0 T}]]
\end{equation}
Denote
$c_n:=\inf_{z\in C_{n}}\inf_{t\in[0,1]} E_{z}[1_{T\in[t+1, t+2]}1_{T<\tau_n}e^{\lambda_0 T}].$
Since $y_n=\gamma_n y_1$, $c_n=c_1$ and $c_1$ is satisfies \eqref{zeros7}. 
By \eqref{eq:61}, \eqref{zeros7}, $C:=c_1C_H$ satisfies \eqref{zeros6}.

 By \eqref{zeros4}, we have the first inequality of \eqref{zeros5}. Using \eqref{zeros6}, we have the second inequality of \eqref{zeros5}. The last inequality of \eqref{zeros5} follows from Harnack inequality \eqref{harnack4}.
 \begin{equation}\label{zeros5} \begin{split}
G_{\lambda_0}(x,y')
&\overset{\eqref{zeros4}}\geq\sum_{k=0}^\infty\sum_{n=1}^\infty \int_{k}^{k+1}\mathbb{E}_x[{1}_{T\in [k+2,k+3)}1_{\omega(T)\in \partial B_n}e^{\lambda_0 T}G_{\lambda_0}(\omega({T}),y')]ds\\
&\overset{\eqref{zeros6}}\geq\sum_{n=1}^\infty\sum_{k=0}^\infty CG_{\lambda_0}(y_n,y')\mathbb{E}_x\left[\int_k^{k+1}e^{\lambda_0 s}1_{C_n}(\omega(s))ds\right]\\
&=\sum_{n=1}^\infty CG_{\lambda_0}(y_n,y')\mathbb{E}_x\left[\int_0^\infty e^{\lambda_0 s}1_{C_n}(\omega(s))ds\right]\\
&=C\sum_{n=1}^\infty G_{\lambda_0}(y_n,y')\int_{C_n}G_{\lambda_0}(x,z)d\mu(z)\\
&\overset{\eqref{harnack4}}\geq CC_H^2\sum_{n=1}^\infty G_{\lambda_0}^2(y_n,x)\mu({C_n})>CC_H^2\sum_{n=1}^\infty c^2=\infty.
\end{split}
\end{equation}
This is a contradiction. Hence, the ${\lambda_0}$-Green function $G_{\lambda_0}(x,y)$ converges to zero as $y$ goes to infinity. Since $G_\lambda(x,y)\leq G_{\lambda_0}(x,y)$ for any $x,y\in \widetilde{X}$, we complete the proof.
 \end{proof}

Let $x_0$ be a point of $\X$ and let $\{y_n\}$ and $\{y_n'\}$ be sequences converging to a point $\xi$ of the Gromov boundary of $\X$. By Theorem \ref{MB1}, for all $x\in\X$, the functions $K_\lambda(x_0,x,y_n)$ and $K_\lambda(x_0,x,y_n')$ converge pointwise to the same function $K_\lambda(x_0,x,\xi)$.
The map from the Gromov boundary to the $\lambda$-Martin boundary is defined by
$$\xi\mapsto K_\lambda(x_0,x,\xi).$$
For two different points $\xi_1,\xi_2$ in the Gromov boundary, let $g_1$ and $g_2$ be the geodesic rays that converge to $\xi_1$ and $\xi_2$, respectively. Let $g_3$ is the geodesic line from $\xi_1$ to $\xi_2$. Since $\X$ is hyperbolic, there exist a point $p$ and a constant $C$ such that for all $i\in\{1,2,3\}$, $p$ is in the $C$-neighborhood of $g_i$. By Harnack inequality and Ancona inequality, for sufficiently large $t>0$, $$K_\lambda(x_0, g_1(t),\xi_2)=O(G_\lambda(p, g_1(t)))\text{ and } K_\lambda(x_0,g_2(t),\xi_2)=O(G_\lambda^{-1}(x_0,g_2(t))).$$ By Lemma \ref{Zero}, $\displaystyle \lim_{t\rightarrow \infty}K_\lambda(x_0, g_1(t),\xi_2)=0$ and $\displaystyle \lim_{t\rightarrow \infty}K_\lambda(x_0, g_2(t),\xi_2)=\infty.$ Hence, two distinct points in the Gromov boundary converge to the distinct points in the $\lambda$-Martin boundary. 
\begin{thm}For any $\lambda\in [0,\lambda_0]$, the Gromov boundary coincides with the $\lambda$-Martin boundary.
\end{thm}
\begin{proof} Suppose that a sequence $\{y_n\}$ in $ \widetilde{X}$ converges to a function $K_\lambda(x_0,x,\zeta)$ of the $\lambda$-Martin boundary. Let us consider the geodesic $g_n$ from $x_0$ to $y_n$. By Arzel\`a-Ascoli's theorem (\cite{BH} Theorem I.3.10), for any integer $m$, the sequence of geodesics $g_n|_{[0,m]}$ has a subsequence that converges to a geodesic of length $m$. By the induction on the lengths of geodesics, we have a subsequence of $g_{n_k}$ that converges to a geodesic ray $g$. Let $\xi$ be a point satisfying $\xi=\displaystyle \lim_{t\rightarrow \infty} g(t)$. Then the subsequence $\{y_{n_k}\}$ converges to $\xi$. Since the subsequence $\{K_\lambda(x_0,x,y_{n_k})\}$ converges pointwise to $K_\lambda(x_0,x,\xi)$, $K_\lambda(x_0,x,\xi)=K_\lambda(x_0,x,\zeta)$ for all $x\in\X$. Hence, the map from the Gromov boundary to $\lambda$-Martin boundary is surjective.
\end{proof}
\appendix
\section{Dirichlet forms and the heat kernels on precompact open sets}\label{appendix}
\addcontentsline{toc}{section}{Appendices}
\renewcommand{\thesubsection}{\Alph{subsection}}

In this appendix, we recall the definitions related to the Dirichlet form on the $L^2$-space of a graph $(\X, d\mu)$ and the construction of the heat kernel on a precompact connected open set $O$ of $\X$. 
\begin{defn} Let $\mathcal{H}=L^2(\X,d\mu)$ with the standard $L^2$-inner product $\left\langle \cdot,\cdot \right\rangle$. For a dense subspace $Dom(\mathcal{E})$ of $\mathcal{H}$, the map $\mathcal{E}:Dom(\mathcal{E}) \times Dom(\mathcal{E}) \rightarrow \mathbb{R}$ is a \textit{symmetric form} if the following properties hold:
 for all $u,v,w \in Dom(\mathcal{E})$ and $\alpha \in \mathbb{R}$,
\begin{align*}
\mathcal{E}(\alpha u+v,w)&= \alpha \mathcal{E}(u,w)+\mathcal{E}(v,w)\\
\mathcal{E}(u,u)&\geq 0\\
\mathcal{E}(u,v)&=\mathcal{E}(v,u).
\end{align*}
\end{defn}

Let $(\mathcal{E},Dom(\mathcal{E}))$ be a symmetric form. For any $\alpha>0$, we define another symmetric form $\mathcal{E}_\alpha$ as follows:
$$\mathcal{E}_\alpha(u,v):=\mathcal{E}(u,v) +\alpha \left\langle u,v \right\rangle \text{ for all } u,v \in Dom(\mathcal{E})\ \ \text{   and }$$ $$Dom(\mathcal{E}_\alpha)=Dom(\mathcal{E}),$$
In particular, when $\alpha=1$, we call $\sqrt{\mathcal{E}_1(u,u)}$ the $\mathcal{E}_1$-norm of a function $u$ in $Dom(\mathcal{E})$.

 For all $u,v\in \mathcal H$, denote by $\wedge$ and $\vee$ the \emph{minimum} and the \emph{maximum functions}:
$$u\wedge v(x)=\min \{u(x),v(x)\}\ \ \text{ and}\ \ u \vee v(x)=\max \{u(x),v(x)\}.$$
\begin{defn}[Dirichlet form]
 Let $\mathcal{E}$ be a symmetric form with domain $ Dom(\mathcal{E})$ contained in $\mathcal{H}$. 
\begin{enumerate}
 \item[(1)] Let $C_0(\X)$ be the space of continuous functions on $\X$ that vanish at infinity. A subspace $\mathcal{C}$ of $Dom(\mathcal{E}) \cap C_0(\X)$ is a \textit{core} if $\mathcal{C}$ is dense in $Dom(\mathcal{E})$ with $\mathcal{E}_1$-norm and $\mathcal{C}$ is dense in $C_0(\X)$ with uniform norm $||\cdot ||_\infty$. The symmetric form $\mathcal{E}$ is \textit{regular} if $\mathcal{E}$ has a core.

 \item[(2)] A symmetric form $\mathcal{E}$ is \textit{strongly local} if for all compactly supported functions \\
 $u,v \in Dom(\mathcal{E})$, $\mathcal{E}(u,v)=0$, when $v$ is constant on a neighborhood of $supp(u)$.
 \item[(3)] A symmetric form $\mathcal{E}$ is \textit{closed} if for any sequence of functions $u_n$ in $Dom(\mathcal{E})$ satisfying $ \displaystyle\lim_{m,n\rightarrow \infty}\mathcal{E}_1(u_n -u_m,u_n-u_m)=0$,
there exists a function $u$ in $Dom(\mathcal{E})$ such that $\displaystyle\lim_{n\rightarrow\infty} \mathcal{E}_1(u_n-u,u_n-u) =0.$
\item[(4)] A closed symmetric form $\mathcal{E}$ is \textit{Markovian} if the following hold: for all $u$ in $Dom(\mathcal{E})$, if $v=(0\wedge u)\vee 1$, then $\mathcal{E}(v,v) \leq \mathcal{E}(u,u).$
\item[(5)] A symmetric form $\mathcal{E}$ is a \textit{Dirichlet form} if $\mathcal{E}$ is closed and Markovian.
\end{enumerate}\end{defn} 
Using the discreteness of the spectrum of $(\Delta,Dom_O(\Delta))$ and the smoothness of the eigenfunctions of $(\Delta,Dom_O(\Delta))$, we will construct a heat kernel on $O$.
 \begin{thm} \emph{(\cite{Sc} Theorem A.3)} The spectrum $\sigma(A)$ of a compact operator $A$ on a Hilbert space $\mathcal{H}$ is at most countable and has no nonzero accumulation point. If the dimension of $\mathcal{H}$ is infinite, $0\in \sigma(A)$. The eigenspace of any eigenvalue $\lambda\neq0$ of $A$ is finite dimensional. 
 \end{thm}
 By the definition of the Laplacian and Cauchy inequality, for any $f\in Dom_O(\Delta)$, $$||(-\Delta+I)f||_{L^2(O)}\geq ||f||_{L^2(O)}.$$ Proposition 2.1 in \cite{Sc} shows that $(-\Delta+I)^{-1}$ is a bounded operator. Since the embedding $\iota:W_0^1(O) \rightarrow L^2(O)$ is a compact operator, $\iota \circ (-\Delta+I)^{-1}$ from $L^2(O)$ to $L^2(O)$ is compact. 

\begin{coro}\label{Rell} Let $O$ be a precompact connected open subset of $ \widetilde{X}$. The spectrum of the Laplacian $(\Delta,Dom_O(\Delta))$ is discrete and every eigenspace is finite dimensional.
\end{coro}

Let $Y$ and $Z$ be normed vector spaces. A function $A$ from an open set $U$ of $Y$ to $Z$ is \emph{Fr\'echet differentiable} if for all $y\in U$, there exists a bounded linear operator $DA(y)$ from $Y$ to $Z$ such that 
$$\lim_{t\rightarrow 0} \frac{||A(y+t)-A(y)-DA(y)t||_{Z}}{||t||_{Y}}=0.$$
The operator $DA(y)$ is \emph{Fr\'echet derivative} at $y$. The function $A$ is a $C^1$-function if Fr\'echet derivative $DA$ on $U$ is continuous.
\begin{thm}\emph{(\cite{Mc} 7.2.a Theorem 1, Lagrange multiplier for Banach space)} Let $Y$ be a Banach spaces and let $A,B:Y \rightarrow \mathbb{R}$ be $C^1$-functions. Denote by $DA$ and $DB$ Fr\'echet derivatives. If $f\in B^{-1}(0)$ is an extreme point of $A$ and $DB(f)$ is a nontrivial linear functional, then there exists a Lagrange multiplier $\lambda\in\mathbb{R}$ such that  
$$DA(f)=\lambda DB(f).$$
\end{thm}

Let $O$ be a precompact connected open set in $\X$. Define functions $$A(f)=||f'||^2_{L^2(\X)},\qquad B(f)=||f||_{L_2(\X)}^2-1$$ on $W_0^1(O)$. Fr\'echet derivatives of  $A$ and $B$ are $$DA(f)g=2\mathcal{E}(f,g)\text{ and }DB(f)g=2(f,g), \text{ respectively.}$$ 

Suppose that $A(f)$ is an extreme value of $A$ on $B^{-1}(0)$.
 Since Fr\'echet derivative $DB(f)$ is nontrivial, there exists a constant $\lambda$ such that for all $g \in W_0^1(O)$,
 $$\mathcal{E}(f,g)=\lambda (f,g).$$
 By Cauchy inequality, $f$ is in $Dom_O(\Delta)$ and $f$ is an eigenfunction of $-\Delta$ with eigenvalue $\lambda$. Using this fact, we find an orthonormal basis of $L^2(O)$ which consists of the eigenfunctions of $\Delta$: 
 \begin{thm}\emph{(\cite{Mc} 7.2.b)}\label{obasis} Let $O$ be a precompact connected open subset of $ \widetilde{X}$. There is a maximal orthonormal set $\{p^{O}_{i}\}$ of eigenfunctions of $-\Delta$ satisfying the following properties: for all $f \in L^2(O)$,
\begin{equation}\label{basis} f=\displaystyle \sum_{i=1}^{\infty} \left\langle f,p^{O}_{i}\right\rangle p^{O}_{i}\end{equation}
and the sequence of the eigenvalues $\lambda^O_i$ corresponding to the eigenfunctions $p^O_i$ is increasing.
\end{thm}
 Let $O$ be a precompact open subset of $ \widetilde{X}$. Using the eigenfunctions $p_i^O$ of $(-\Delta,Dom_O(\Delta))$ with eigenvalues $\lambda_i^O$, define a function $p_O:(0,\infty) \times O \times O \rightarrow \mathbb{R}$ as follows:
\begin{equation}\label{A.2}p_O(t,x,y) =\displaystyle \sum e^{-\lambda^{O}_it}p^{O}_{i}(x)p^{O}_{i}(y).\end{equation}

Similar to the proof of Lemma \ref{bwws2}, we also have the following lemma.
 \begin{lem}\label{bwws3}
Let $O$ be a precompact open subset of $ \widetilde{X}$. The eigenfunctions of $(\Delta,Dom_O(\Delta))$ are contained in $D_c^\infty(\X)$.
\end{lem}

 \begin{lem} \label{Maximum} Let $I=[0,T]\subset \mathbb{R}$ and let $O$ be a precompact connected open subset of $ \widetilde{X}$. Let $u:I \times \bar{O} \rightarrow \mathbb{R}$ be a function such that $u(\cdot,x):I^o \rightarrow \mathbb{R}$ is differentiable for any $x \in O$ and $u(t,\cdot)$ is in $D_c^\infty(\bar{O})$ for any $t \in I^o$. If 
\begin{equation}\label{heat1}
 \Delta u(t,x)-\frac{d}{dt}u(t,x) \geq 0 \emph{ for all } (t,x) \in I^o \times O,
\end{equation}
then 
\begin{equation}\label{max}
\max_{I\times \overline{O}} u=\max_{\substack{I\times \partial O\cup\\ \{0\}\times O}} u.
\end{equation}
\end{lem}
\begin{proof} Fix $\epsilon>0$. Denote $w_\epsilon=u-\epsilon t.$ Since $w_\epsilon\in D_c^\infty(O),$ We have
 \begin{equation}\label{eq:67}
 \Delta w_\epsilon(t,x)-\frac{d}{dt}w_\epsilon(t,x)= \Delta u(t,x)-\frac{d}{dt}u(t,x)+\epsilon>0.
 \end{equation}
 Suppose that there exists a point $(t_0,x_0)\in (0,T]\times O$ such that $w_\epsilon|_e(t_0,x_0)$ is the maximum of $w_\epsilon$. If $x_0\in e^o$ for some $e$, then $\partial w_\epsilon|_e(t_0,x_0)=0,$ where as if $x_0=i(e)$ ($x_0=i(e)$, resp.) for some $e$, $\partial w_\epsilon|_e(t_0,x_0)\leq0$ ($\partial w_\epsilon(t_0,x_0)\geq0,$ resp.). Since $w_\epsilon$ satisfies the Kirchhoff's law, $\partial w_\epsilon|_e(t_0,x_0)=0$ for any $e$ containing $x_0$. We also have $\frac{d}{dt}w_\epsilon(t_0,x_0)\geq 0$ and $\Delta w_\epsilon (t_0,x_0)\leq 0.$ This contradicts to \eqref{eq:67}. This implies that if $w_\epsilon(t_0,x_0)$ is maximum, then $(t_0,x_0)\in I\times \partial O\cup \{0\}\times O.$
 Since $$\max_{I\times \overline{O}}u\leq \max_{I\times \overline{O}}(w_\epsilon+\epsilon t)\leq \max_{I\times \overline{O}}(w_\epsilon +\epsilon T)=\max_{\substack{I\times \partial O\cup\\ \{0\}\times O}}(w_\epsilon +\epsilon T)$$ and $\epsilon$ is arbitrary, we have \eqref{max}.
\end{proof}

 Since $\G$ is non-amenable and the bottom of the spectrum is non-zero by Theorem 8.5  in \cite{SW}, $\lambda_{O,i}>0$. The function $p_O(t,x,y)$ satisfies the assumption of Lemma \ref{Maximum}. 
\begin{prop}\label{Aa} For any precompact connected open set $O$, the function $p_O(t,x,y)$ satisfies the following:
\begin{enumerate}
\item[(1)] $p_O(t,x,y)=p_O(t,y,x)$ and $\displaystyle \frac{d}{dt}p_O(t,x,y) =\Delta_y p_O(t,x,y),$
\item[(2)] $\displaystyle \int_O p_O(t,x,z)p_O(s,z,y)d\mu(z)=p_O(t+s,x,y)$ for all $x,y \in O$,
\item[(3)] $p_O(t,x,y) >0$ for all $x,y \in O$ and for all $t>0$,
\item[(4)] $\displaystyle \int_O p_O(t,x,y)d\mu(y) \leq 1$ for all $t \geq 0.$ 
\end{enumerate}
\end{prop}
\begin{proof}
As the proof of Lemma 3.2 in \cite{D}, the parts $(1)$, $(2)$ and $(3)$ are proved by the construction of $p_O(t,x,y)$ and Lemma \ref{Maximum}. Since $\underset{t\rightarrow 0}\lim\int_Op_O(t,x,y)d\mu(y)=1$, it remains to show that 
$\frac{d}{dt}\int_O p_O(t,x,y)d\mu(y) \leq0.$

Denote by $E_y$ the set of edges containing $y$. Since $p_O(t,x,\cdot)$ satisfies Kirchhoff's law, we obtain the following:
\begin{align*}
\displaystyle\int_O \frac{d}{dt}p_O(t,x,y)d\mu(y)&=\displaystyle \int_O \Delta_y p_O(t,x,y)d\mu(y)\\
&= \displaystyle \sum_{y\in \partial O}\left\{\sum_{\substack {e\in E_y\\i(e)\in O}} \frac {dp_O|_e}{dy}(t,x,y)- \sum_{\substack{e\in  E_y\\ t(e)\in O}} \frac {dp_O|_e}{dy}(t,x,y)\right\}.
\end{align*}
 Since for all $(t,x,y) \in (0,\infty) \times O \times O$, $p_O(t,x,y)>0$, and $p_O(t,x,y)$ vanishes on the boundary of $O$,  $\frac {dp_O|_e}{dy}(t,x,y))\leq0$ when $e\in E_y$ and $i(e)\in O$ and $\frac {dp_O|_e}{dy}(t,x,y))\geq0$ when $e\in E_y$ and $t(e)\in O$. Hence, $(4)$ holds.
\end{proof}
 \begin{defn}\label{heatk}Let $\{O_i\}_{i \geq 1}$ be an increasing sequence of precompact connected open subsets such that $\overline{O_i} \subset O_{i+1}$ for all $i$ and $\displaystyle \bigcup\limits_{i=1}^\infty \overline{O_i}= \widetilde{X}$. 
 Since $p_{O_i}(t,x,y)\leq 1$ for all $(t,x,y)$ and for all $i$,  we define the function $p(t,x,y)$ as follows:
$$p(t,x,y)=\lim_{i\rightarrow \infty}p_{O_i}(t,x,y).$$
\end{defn}
Note that  $$p(t,x,y)=\sup_{O\in \mathcal{O}} p_O(t,x,y),$$
 where the supremum is taken over the set $\mathcal{O}$ of all precompact connected open subsets of $ \widetilde{X},$
 since given a precompact connected open set $O$, there exists a precompact connected open set $O_i$ such that $O \subset O_i$.
\section*{Acknowledgements} We would like to thank F. Ledrappier for his encouragement and helpful suggestions. This work was supported by Samsung Science and Technology Foundation under Project Number SSTF-BA1601-03 and the National Research Foundation of Korea under Project Number NRF-2020R1A2C1A1011543. The first author was supported by BK21 PLUS SNU Mathematical Sciences Division.

\medskip
Received xxxx 20xx; revised xxxx 20xx.
\medskip


\begin{thebibliography}{99}
\bibitem {A} 
\newblock A. Ancona, 
\newblock {Negatively curved manifolds, elliptic operators and the Martin boundary}, 
\newblock \emph{Ann. of Math.}, \textbf{125} (1987), 495--536.
  
  \bibitem {AR} 
 \newblock S. Albeverio and M. Rockner, 
 \newblock {Classical Dirichlet forms on topological spaces}-{the construction of
an associated diffusion process,} 
\newblock \emph{Probab. Th. Rel. Fields}, \textbf{83} (1989), 405--434.

 \bibitem {Be} 
 \newblock A. F. Beardon, 
 \newblock \emph{The Geometry of Discrete Groups}, 
 \newblock Graduate Texts in Mathematics, \textbf{91}, Springer-Verlag, New York, 1995.
 
 \bibitem{Bo} 
 \newblock P. Bougerol, 
 \newblock{Th\'eor\`eme central limite local sur certains groupes de Lie},
 \newblock  \emph{Ann. Sci. \'Ecole Norm. Sup.}, \textbf{14} (1981), 403--432.

 \bibitem{BH} 
 \newblock M. Bridson and A. Haefliger, 
 \newblock \emph{Metric Spaces of Non-Positive Curvature}, 
 \newblock Fundamental Principles of Mathematical Sciences, 319. Springer-Verlag, Berlin, 1999.
 
  \bibitem{BK} 
  \newblock M. Brin and Y. Kifer, 
  \newblock {Brownian motion, harmonic functions and hyperbolicity for Euclidean complexes,}
  \newblock \emph{Math. Z.}, \textbf{237} (2001), 421-468.
  
  \bibitem{BS} 
  \newblock M. Bonk and O. Schramm, 
  \newblock {Embeddings of Gromov hyperbolic spaces}, 
  \newblock \emph{Geom. Funct. Anal.}, \textbf{10} (2000), 266--306.

 \bibitem{BSSW} 
 \newblock A. Bendikov, L. Saloff-Coste, M. Salvatori and W. Woess, 
 \newblock {The heat semigroup and Brownian motion on strip complexes}, 
 \newblock \emph{Adv. in Math.}, \textbf{226} (2011), 992--1055.
 
\bibitem{CY} 
\newblock S.Y. Cheng and S. T. Yao,
\newblock {Differential equations on Riemannian manifolds and their geometric applications}, 
\newblock \emph{Comm. on Pure Appl. Math.}, \textbf{28} (1975), 333--354.

\bibitem{D} 
\newblock J. Dodziuk, 
\newblock Maximum principle for parabolic inequalities and the heat flow on open manifolds, 
\newblock \emph{Indiana Univ. Math. J.}, \textbf{32} (1983), 703--716. 

\bibitem{EF} 
\newblock J. Eells and B. Fuglede, 
\newblock \emph{Harmonic Maps, Between Riemannian Polyhedra}, 
\newblock Cambridge Tracts in Mathematics, \textbf{142} Cambridge University Press, Cambridge, 2001.

\bibitem{F} 
\newblock M. Fukushima, 
\newblock \emph{Dirichlet Forms and Markov Processes}, 
\newblock North-Holland, Amsterdam and Tokyo,1980.

 \bibitem{FOT} 
\newblock M. Fukushima, Y. Oshima and M. Takeda, 
\newblock \emph{Dirichlet Forms and Symmetric Markov Process, Second revised and extended edition}, 
\newblock De Gruyter Studies in Mathematics, \textbf{19} Walter de Gruyter, Berlin, 2010.

 \bibitem{G} 
  \newblock S. Gou\"ezel, 
 \newblock {Local limit theorem for symmetric random walks in Gromov-hyperbolic groups,},  
 \newblock \emph{J. Amer. Math. Soc.}, \textbf{27} (2014) 893--928.
 
 \bibitem{Gr} 
 \newblock A. Grigor'yan, 
 \newblock{Heat kernels and function theory on metric measure spaces}, 
 \newblock \emph{Contemp. Math.}, \textbf{338} (2003), 143--172.
 
 \bibitem{GH} 
 \newblock \'E. Ghys and P. de la Harpe, 
 \newblock \emph{Sur les Groupes Hyperboliques d'apr$\grave{e}$s Mikhael Gromov}, 
 \newblock Progress in Mathematics, 83, Birkh\"auser Boston, Boston, MA, 1990.
 
\bibitem{GL} 
\newblock S. Gou\"ezel and P. Lalley, 
\newblock{Random walks on co-compact Fuchsian groups}, 
\newblock \emph{Ann. Sci. \'Ecole Norm. Sup.}, \textbf{46} (2013), 129--173.

\bibitem{HK} 
\newblock S. Haeseler and M. Keller, 
\newblock{Generalized solutions and spectrum for Dirichlet forms on graphs,} 
\newblock \emph{in Random walks, boundaries and spectra}, Progr. Probab., Birkh\"auser/Springer Basel AG, Basel, \textbf{64}, (2011) 181--199.

\bibitem{KL} 
\newblock M. Keller and D. Lenz, 
\newblock {Dirichlet forms and stochastic completeness of graphs and subgraphs}, 
\newblock \emph{J. Reine. Angew. Math.}, \textbf{666} (2012),189--223.

 \bibitem{KPS}
 \newblock V. Kostrykin, J. Potthofff and R. Schrader, 
 \newblock {Brownian motions on metric graphs}, 
 \newblock \emph{J. Math. Phys.}, \textbf{53} (2012) 36pp. 

\bibitem{KS} 
\newblock V. Kostrykin and R. Schrader, 
\newblock {Laplacians on metric graphs: Eigenvalues, resolvents and semigroups,} 
\newblock \emph{in Quantum Graphs and Their Applications}, (edited by G. Berkolaiko, R. Carlson, S. A. Fulling, and P. Kuchment), Contemp. Math., Amer. Math. Soc., Providence, RI, \textbf{415} (2006) 201–225.

\bibitem{LL}
\newblock F. Ledrappier and S. Lim,  
\newblock {Local limit theorem in negative curvature}, 
\newblock to appear Duke Mathematics Journal., arxiv{1503.04156}.

\bibitem{LS} 
\newblock T. Lyons and D. Sullivan, 
\newblock {Function theory, random paths and covering spaces}, 
\newblock \emph{J. Differential Geom.}, \textbf{19} (1984) 299--323.

\bibitem{LSV}
\newblock D. Lenz, P, Stollmann and I. Veseli\'c, 
\newblock {The Allegretto-Piepenbrink theorem for strongly local Dirichlet forms,} 
\newblock \emph{Documenta Math.}, \textbf{14} (2009) 167--189.

 \bibitem{Mc}
 \newblock R. McOwen, 
 \newblock \emph{Partial Differential Equations: Methods and Applications}, 
 \newblock Prentice Hall, Upper Saddle River, NJ,  1996. 
 
 \bibitem{Mu}
  \newblock J. R. Munkres, 
 \newblock \emph{Topology. Second edition}, 
 \newblock Prentice Hall, Inc., Upper Saddle River, NJ, 2000.
 
\bibitem{PS} 
\newblock M. Pivarski and L. Saloff-Coste, 
\newblock {Small time heat kernel behavior on Riemannian complexes,} 
\newblock \emph{New York J. Math.}, \textbf{14} (2008) 459--494.

 \bibitem{Sc} 
 \newblock K. Schm\"udegen, 
 \newblock \emph{Unbounded Self-adjoint Operators on Hilbert Space}, 
 \newblock Graduate Texts in Mathematics, 265, Springer, Dordrecht, 2012.
 
\bibitem{Si} 
\newblock M. L. Silverstein, 
\newblock \emph{Symmetric Markov Processes}, 
\newblock Lecture Notes in Mathematics No. \textbf{426}, Springer, Berlin, Heidelberg, New York, 1974

 \bibitem{St1} 
  \newblock K-T Sturm, 
 \newblock {Analysis on local Dirichlet spaces-I . Recurrence, conservativeness and $L^p$-Liouville properties.} 
 \newblock \emph{J. Reine Angew. Math.}, \textbf{456} (1994) 173--196.
 
 \bibitem{St2} 
 \newblock K-T Sturm, 
 \newblock {Analysis on local Dirichlet spaces-II. Upper Gaussian estimates for the fundamental solutions of parabolic equations,} \newblock \emph{Osaka J. Math.}, \textbf{32} (1995) 275--312.
 
\bibitem{St3} 
\newblock K-T Sturm, 
\newblock {Analysis on local Dirichlet spaces-III. The parabolic Harnack inequality,} 
\newblock \emph{J. Math. Pures Appl.}, \textbf{75} (1996), no. 3,  273--297.

\bibitem{St4}
\newblock K-T Sturm, 
\newblock {Metric measure spaces with variable Ricci bounds and couplings of Brownian motions,}  
\newblock in \emph{Festschrift Masatoshi Fukushima}, (edited by Z.-Q. Chen, N. Jacob, M. Takeda and T. Uemura), Interdiscip. Math. Sci, World Sci., World Sci. Publ., Hackensack, NJ, \textbf{17} (2015) 553--575.

\bibitem{Su} 
\newblock D. Sullivan, 
\newblock{Related aspects of positivity in Riemannian geometry}, 
\newblock \emph{J. Differential Geom.}, \textbf{25} (1987) 327--351. 

\bibitem{SW} 
\newblock L. Saloff-Coste and W. Woess, 
\newblock {Transition operators on co-compact G-spaces}, 
\newblock \emph{Rev. Mat. Iberoam.}, \textbf{22} (2006) 747--799.

\bibitem{SW2} 
\newblock L. Saloff-Coste and W. Woess, 
\newblock {Computations of spectral radii on $\mathcal G$-spaces}, 
\newblock in \emph{Spectral Analysis in Geometry and Number Theory} (edited by M. Kotani, H. Nalto and T. Tate), {Contemp. Math.,} \textbf{484}, (2009)195--218.

\bibitem{W} 
\newblock R.K. Wojciechowski, 
\newblock {Heat kernel and essential spectrum of infinite graphs}, 
\newblock \emph{Indiana Univ. Math. J.}, \textbf{58} (2009), no. 3, 1419--1441.

\end{thebibliography}
\end{document}